\documentclass[11pt]{article}
\usepackage{amsmath}
\usepackage{amssymb}
\usepackage{amsthm}
\usepackage{enumerate}
\usepackage{epsfig}
\usepackage{dsfont}
\usepackage{color}
\usepackage{graphicx}
\usepackage[a4paper]{geometry}
\usepackage[utf8]{inputenc}

\newcommand{\N}{\ensuremath{\mathbb{N}}}
 
\newcommand{\Op}{\ensuremath{\mathcal{L}}}
\newcommand{\Hyp}{\ensuremath{\mathbb{H}}}
\newcommand{\R}{\ensuremath{\mathbb{R}}}
\newcommand{\Prob}{\ensuremath{\mathbb{P}}}

\newcommand{\E}{\ensuremath{\mathbb{E}}}

\newcommand{\AdS}{\mathrm{AdS}^{d+1}}
\newcommand{\PR}{\mathrm{P}^{d+1} \mathbb{R}}
\newcommand{\Ein}{\mathrm{Ein}^{d}}
\newcommand{\dS}{\mathrm{dS}^{d+1}}

\newtheorem{lemm}{Lemma}
\newtheorem{theo}{Theorem}
\newtheorem{prop}{Proposition}
\newtheorem{coro}{Corollary}
\newtheorem*{hypo}{The d\'evissage conditions}
\newtheorem{rema}{Remark}

\author{J\"urgen Angst\footnote{Univ Rennes, CNRS, IRMAR - UMR 6625, F-35000 Rennes, France, jurgen.angst@univ-rennes1.fr} \; and Camille Tardif\footnote{Sorbonne Universit\'e - LPSM, UMR 8001, 75252 Paris CEDEX 05, camille.tardif@sorbonne-universite.fr}}

\begin{document}

\title{On the Poisson boundary of the relativistic Brownian motion}
\maketitle

\abstract{
In this paper, we determine the Poisson boundary of the relativistic Brownian motion in two classes of Lorentzian manifolds, namely model manifolds of constant scalar curvature and Robertson--Walker space-times, the latter constituting a large family of curved manifolds. Our objective is two fold: on the one hand, to understand the interplay between the geometry at infinity of these manifolds and the asymptotics of random sample paths, in particular to compare the stochastic compactification given by the Poisson boundary to classical purely geometric compactifications such as the conformal or causal boundaries.  On the other hand, we want to illustrate the power of the d\'evissage method introduced by the authors in \cite{devissage}, method which we show to be particularly well suited in the geometric contexts under consideration here. 
}

\par
\vspace{1cm}
\setcounter{tocdepth}{3}
\tableofcontents

\newpage

\section{Introduction and setting}

\subsection{Long time asymptotics of relativistic Markov processes}

\label{sec.intro}
The study of stochastic processes in the context of Lorentzian geometry has both physical and mathematical strong motivations. One the one hand, from a physical perspective, Lorentzian diffusions are very adequate models to study random motions or fluid dynamics in the framework of Einstein's special or General Relativity theory, see for example \cite{dmr,dunkel,deb} and the references  therein. On the other hand, from a more mathematical perspective, considering the importance of the heat kernel as a tool in Riemannian geometry, it appears very natural to investigate the links between local/global geometry and the asymptotics of random paths in a Lorentzian setting. 
\par 
\medskip
Among the natural questions concerning the interplay between randomness and geometry, we are here particularly interested in understanding how the long time asymptotic behavior of random processes reflects the geometry at infinity of the underlying manifold on which they are defined. The study of the asymptotic geometry of a manifold identifies with the one of its geometric compactifications. In the Lorentzian framework,
due to their importance in physics and in particular in cosmology, such compactifications of space-times have been introduced and intensively studied by both physicists and geometers. The most popular constructions in this context are the conformal and causal boundaries, see e.g. \cite{hawell} and the references therein.
\par 
\medskip
From both geometric and probabilistic points of view, of primer interest are the so-called relativistic processes, that is to say random processes whose law is covariant under the action of local isometries of the manifold. In the case where the latter has constant curvature, relativistic Markov processes can be entirely classified. They are the projections of invariant L\'evy processes with values in the isometry group of the manifold, see \cite{camille}. If the base manifold has non-constant curvature, jumps of the sample paths are difficult to handle and one usually restricts the study to the one of continuous Markov processes, i.e. diffusion processes.
Following Dudley's seminal work \cite{dudley1,dudley2} in Minkowski space-time, Franchi and Le Jan constructed in \cite{flj}, on the future-directed half of the unitary tangent bundle of an arbitrary Lorentz manifold, a diffusion process which is Lorentz-covariant. This process, that we will simply call \textit{relativistic Brownian motion} or \textit{relativistic diffusion} in the sequel, and whose precise definition is recalled in Section \ref{sec.RWdiffusion} below, is the Lorentzian analogue of the classical Brownian motion on a Riemannian manifold. It can be seen either as a random perturbation of the timelike geodesic flow on the unitary tangent bundle, or as a stochastic development of Dudley's diffusion in a fixed tangent space. Variants and generalizations of this process can be found in \cite{flj2,debbaschRD,ismaelRD}. 
\par
\medskip
The study of the articulation between the asymptotics of relativistic processes and the geometry at infinity of the underlying manifold can be performed by following the general following scheme. A first task consists in expliciting the almost sure long time asymptotics of the sample paths of the process. Beyond this almost sure asymptotics, the long time behavior of the process is then fully encoded in its Poisson boundary, whose definition is also recalled in Section \ref{sec.RWdiffusion} below, and which captures all the probabilistic ``information" at infinity. Having identified the Poisson boundary, one can then compare its support to the purely geometric boundaries, see which one carries more information etc. Let us briefly recall here that the determination of the Poisson boundary can also be rephrased in the language of harmonic analysis since it identifies with the set of bounded harmonic functions for the infinitesimal generator of the process. 
\par 
\medskip
In a general setting, the explicit determination of the Poisson boundary of a Markov process is a highly non-trivial task. Indeed, as soon as the underlying state space is not a semi-simple Lie group or an homogeneous space, standard Lie groups methods such as the ones developped in \cite{harry1,harry,liao,babillot}, techniques such as explicit Doob $h-$tranforms \cite{pinsky},  explicit couplings or shift-couplings \cite{cranston} are hardly implementable. Recently, the authors introduced in \cite{devissage} the so-called d\'evissage method which allows to overcome this difficulty, at least in the case where the state space has enough symmetries and the dynamics of the considered process respects these symmetries. 
\newpage
Following the general scheme described above, a detailed study of the long time almost sure asymptotics of the relativistic diffusion has been performed in a certain number of examples of Lorentzian manifolds: Minkowski space-time \cite{dudley3}, Schwartzschild space-time \cite{flj}, G\"odel space-time \cite{franchigoedel}, de Sitter and Anti de Sitter space-times in \cite{camille} and in a large class of curved, warped product space-times in \cite{mathese, angst2016}. 
Nevertheless, the full determination of the Poisson boundary of a relativistic Markov process has only  been performed in the continuous context, namely for the relativistic diffusion, and only in Minkowski space-time in \cite{ismael,devissage} and in a very particular case of Robertson--Walker space-time in \cite{angst1}. In these two examples, it was shown that the Poisson boundary of the relativistic diffusion actually identifies with the geometric conformal and causal boundaries cited above, so that it is quite natural to ask if it is always the case.
\par 
\medskip
In this work, thanks to an extensive use of the d\'evissage method, we explicitly determine the Poisson boundary of the relativistic diffusion in a large class of Lorentzian manifolds. This class is composed of all model space-times of constant curvature namely Minkowski, de Sitter and Anti de Sitter space-times, as well as all expanding  Robertson--Walker space-times studied in \cite{angst2016}. Doing so, we are able to explicitly relate the geometry at infinity of the underlying manifold via its conformal/causal boundary to the sample paths asymptotics. In particular, we show that the Poisson boundary of the relativistic diffusion indeed identifies with the conformal boundary in the case of model space-times of non-negative curvature, but that it is no more the case in Anti de Sitter space-time or for curved space-times. In the studied cases, our results confirm a conjecture by Franchi and Le Jan asserting that the Poisson boundary of the relativistic diffusion identifies with equivalence classes of light rays.
\par
\medskip 
Let us emphasize that in each geometric case considered in the article, the exact determination of the Poisson boundary is obtained via the d\'evissage method. Nevertheless, our results cover a large variety of geometric behaviors, each case demanding a specific treatment and the adaptation of the d\'evissage scheme. The fact that the same method allows to conclude in such a variety of situations illustrates its power and flexibility. 
\par
\medskip 
The plan of the paper is the following: in the next subsection, we briefly introduce the geometric and probabilistic settings of our study. In Section \ref{sec.statement}, we then state our main results, which consist in the exact determination of the Poisson boundary of the relativistic diffusion in all the manifolds under consideration. Finally, Section \ref{sec.proof} is devoted to the proofs of the stated results.

\subsection{Geometric and probabilistic settings}

In order to state our results in the next section, we quickly introduce here the geometric and probabilistic frameworks of our study.

\subsubsection{Lorentzian model manifolds}\label{sec.background}

The global geometric framework of our study is the one of Lorentzian manifolds, that is, finite dimension differentiable manifolds $\mathcal M$, endowed with a pseudo-metric $g$ of signature $(-, +, \ldots, +)$.
Due to the non-positivity of the metric, given a point $\xi \in \mathcal M$, a tangent vector $\dot{\xi} \in T_{\xi} \mathcal M$ to such a manifold can be of three different types, namely $\dot{\xi}$ is said to be time-like if $g_{\xi}(\dot{\xi}, \dot{\xi}) <0$, space-like if $g_{\xi}(\dot{\xi}, \dot{\xi}) >0$ and light-like if $g_{\xi}(\dot{\xi}, \dot{\xi})=0$. In the same way, a smooth curve $(\xi_t)_{t \in I}$ on $\mathcal M$ is said to be time-like (resp. space/light-like) if for each $t \in I$, the tangent vector $\dot{\xi}_t = d \xi_t/dt \in T_{\xi_t} \mathcal M$ is time-like (resp. space/light-like). A time-like curve with values in $\mathcal M$ can always be parametrized by its arc-length or proper time $s$, so that $g_{\xi_s}(\dot{\xi}_s, \dot{\xi}_s) =-1$. The unitary tangent bundle associated to time-like tangent vectors of pseudo-norm $-1$ will be denoted by $T^1 \mathcal M$, and if a chronological orientation if given on $\mathcal M$, then $T^1_+ \mathcal M$ will denote its positive part consisting of future oriented vectors.
\newpage
As in the Riemannian setting, a Levi--Civita connection can be associated to the pseudo-metric $g$ which allows to define a Riemann curvature tensor, and thus a scalar curvature after taking the trace. 
Of primer interest are then the Lorentzian manifolds of constant scalar curvature, or model space-times, which are the analogues of the Euclidean space $\mathbb R^d$, the Euclidean sphere $\mathbb S^d$ and the hyperbolic space $\mathbb H^d$ in the Riemannian context. Let $Q_{p,q}$ denote the canonical quadradic form of signature $(p,q)$ on $\mathbb R^{p+q}$, namely for $x=(x_1,\ldots, x_{p+q}) \in \mathbb R^{p+q}$
\[
Q_{p,q}(x,x) := - \sum_{k=1}^p |x_k|^2 + \sum_{\ell=p+1}^q |x_{\ell}|^2.
\]
The Lorentzian manifold of dimension $d+1$ with constant zero scalar curvature is the Minkowski space-time $\mathbb R^{1,d}$ which is simply defined as $\mathbb R^{d+1}$, endowed with pseudo-metric $Q_{1,d}$. The analogue of the Euclidean sphere, i.e. the Lorentzian manifold of constant scalar curvature equal to one is the de Sitter space-time $\mathrm{d S}^{d+1} $ which is defined as 
\[
\mathrm{d S}^{d+1}:= \left \lbrace  x \in \mathbb R^{d+2}, \; Q_{1,d+1}(x,x)= 1\right \rbrace,
\] 
endowed with the metric $Q_{1,d+1}$ inherited of the ambiant space. Finally, the Lorentzian manifold of constant scalar curvature equal to $-1$  is the Anti de Sitter space-time $\mathrm{A d S}^{d+1} $ defined as 
\[
\mathrm{ A d S}^{d+1}:= \left \lbrace x \in \mathbb R^{d+2}, \; Q_{2,d}(x,x) = -1\right \rbrace,
\] 
also endowed with pseudo-metric $Q_{2,d}$ inherited of the ambiant space.
As already mentionned, we will be mostly interested here in relating the long time asymptotic behavior of the relativistic Brownian motion in a space-time $\mathcal M$ to the geometry at infinity of the latter. The three model space-times $\mathcal M=\mathbb R^{1,d}, \mathrm{d S}^{d+1}$ and $\mathrm{A d S}^{d+1} $ are all conformally flat and each admits a natural compactification as a subset of Einstein static universe $\mathbb R \times \mathbb S^d$. This compactification, introduced by Penrose in \cite{penrose}, is called the conformal boundary of the manifold. In the case of $\R^{1,d}$ and $\dS$ this conformal boundary coincides with the causal boundary and is topologically a cone for $\R^{1,d}$ and a sphere for $\dS$. The manifold $\AdS$ is not causal since its contains closed  time-like geodesic and thus, it does not admit a causal boundary. Nevertheless it has a well defined conformal boundary, which is topologically a torus $\mathbb{S}^1 \times \mathbb{S}^{d-1}$, identified with the Einstein flat conformal Lorentz manifold of dimension $d$, denoted by $\text{Ein}^d$ in the sequel. For a detailed description of Lorentzian model space-times and their conformal boundary, we refer to \cite{frances2,frances} and the references therein.

%
\subsubsection{Robertson Walker space-times} \label{sec.geoRW}
The second type of Lorentzian manifolds we will consider in this article is the one of Robertson--Walker space-times. They are among the simplest examples of curved space-times, yet their geometry is rich and flexible enough to have a good idea of the interplay between the random paths asymptotics and the manifold on which they live. A Robertson--Walker space-time $\mathcal M:=I \times_{\alpha} M$ is defined as a Cartesian product of a open interval $(I,-dt^2)$ (the base) and a Riemannian manifold $(M, h)$ of constant curvature (the fiber), endowed with a Lorentz metric of the following form $g := -dt^2 + \alpha^2(t)h,$ where $\alpha$ is a positive $C^2$ function on $I$, called the \textit{expansion function} or \textit{torsion function}. A general study on the geometry of warped product manifolds can be found in \cite{ghani1}. More specific results on the geometry of Robertson--Walker space-times and their geodesics can be found in \cite{floressanchez}. 
Since a Riemannian manifold of constant curvature is isometric to the Euclidean space, its sphere of the hyperbolic space, without loss of generality, we can restrict ourselves to the cases where $M= \mathbb R^d, \mathbb S^d$ or $\mathbb H^d$. We will assume here that the ``time interval'' is infinite i.e. $I=(0, +\infty)$, that the torsion function $\alpha(t)$ goes to $+\infty$ as $t$ goes to infinity, and that it satisfies the natural set of Hypotheses 1 and 2 of  \cite{angst2016}, which can be summarized as 
\begin{itemize}
\item $\alpha$ is $\mathrm{log}-$concave on $I$, i.e. the Hubble function $ H:=\alpha'/\alpha$ is non-increasing ;
\item $\alpha$ has either, polynomial, sub-exponential or exponential growth.
\end{itemize}
All Robertson--Walker space-times are conformally flat, so they also admit a conformal compactification as subsets of Einstein static universe as the model space-times considered above. Nevertheless, this compactification is not intrisic and this is the reason why another compactification is usually prefered in the physical and mathematical litterature: the causal boundary. This compactification, first introduced in \cite{geroch}, consists in attaching an ideal point to every inextensible, time/light-like curve in such a way that the ideal point only depends on the past of the trajectory. The precise definition of the causal boundary and its determination in the case of Robertson--Walker space-times can be found in details in \cite{flores}. In the sequel, the (future oriented component of the) causal boundary of a space-time $\mathcal M$ will be simply denoted by $\partial \mathcal M$.

\subsubsection{Relativistic Brownian motion}
\label{sec.RWdiffusion}

The stochastic process which is the object of our attention in this article is the relativistic diffusion introduced in \cite{flj}. Let us recall that this process is the natural generalization of the standard Riemannian Brownian motion to the Lorentzian context. The sample paths $(\xi_s, \dot{\xi}_s)$ of the relativistic diffusion are time-like curves that are future directed and parametrized by the arc length $s$ so that the diffusion actually lives on the positive part of the unitary tangent bundle $T^1_+ \mathcal M$ of a general Lorentzian manifold $(\mathcal M,g)$ of dimension $d+1$.
These sample paths can be seen either as random perturbations of time-like geodesics or as the stochastic developement of Dudley's Minkowskian diffusion in the initial fixed tangent space, the latter being the unique continuous Markov process whose law is Lorentz covariant, see \cite{dudley1}.
More prosaically, the infinitesimal generator of the diffusion writes 
\begin{equation}\mathcal L:= \mathcal L_0 + \frac{\sigma^2}{2 } \Delta_{\mathcal V},\end{equation}
where the differential operator $\mathcal L_0$ generates the geodesic flow on $T^1 \mathcal M$, $\Delta_{\mathcal V}$ is the vertical Laplacian, and $\sigma > 0$ is a real parameter. Equivalently, if $\xi^{\mu}$ is a local chart on $\mathcal M$ and if $\Gamma_{\nu \rho}^{\mu}$ denote the usual Christoffel symbols, the relativistic diffusion is the solution of the following system of stochastic differential equations (in It\^o form), for $0 \leq \mu \leq d$
\begin{equation}\label{eqn.flj}
\left \lbrace \begin{array}{l}
\displaystyle{ d \xi^{\mu}_s = \dot{\xi}_s^{\mu} ds}, \\
\displaystyle{ d \dot{\xi}^{\mu}_s= -\Gamma_{\nu \rho}^{\mu}(\xi_s)\, \dot{\xi}^{\nu}_s \dot{\xi}^{\rho}_s ds + d \times \frac{\sigma^2}{2}\, \dot{\xi}^{\mu}_s ds+ \sigma  d M^{\mu}_s}, 
\end{array}\right.
\end{equation}
where the braket of the martingales $M^{\mu}_s$ is given by
$$\langle dM_s^{\mu}, \;  dM_s^{\nu}\rangle = (\dot{\xi}^{\mu}_s \dot{\xi}^{\nu}_s +g^{\mu \nu}(\xi_s))ds.$$
Moreover, since the sample paths are parametrized by the arc length $s$, we have the pseudo-norm relation:
\begin{equation}\label{eqn.pseudo}
g_{\mu \nu}(\xi_s) \dot{\xi}^{\mu}_s \dot{\xi}^{\nu}_s =-1.
\end{equation}

\subsubsection{Poisson boundary and the d\'evissage method}\label{sec.devissage}
Let us conclude this introduction by emphasizing that the long time asymptotic behavior of a stochastic process is fully captured by its invariant sigma field, also classically called the Poisson boundary. Recall that a continuous Markov process $(Z_s)_{s \geq 0}$ with values in a state space $E$ can always be realized as the coordinate process $Z_s(\omega)=\omega_s$ on the canonical probability space $(\Omega, \mathcal F)$ where $\Omega:= C(\mathbb R^+, E)$ is the paths space and $\mathcal F$ is its standard Borel sigma field. The asymptotic sigma field of the process is then defined as
$ \mathcal F^{\infty} := \bigcap_{s \geq 0} \sigma( Z_u, \, u >s).$
Considering the classical shift operators $(\theta_t)_{t \geq 0}$ on $\Omega$:
\[ 
\begin{array}{lcll} \theta_s : &  \Omega & \to & \Omega \\ & \omega=(\omega_s)_{s \geq 0} & \mapsto & \theta_t \omega := (\omega_{t+s})_{s \geq 0} \end{array},
\]
the invariant sigma field $\textrm{Inv}((Z_s)_{s \geq 0})$ associated to the process $(Z_s)_{s \geq 0}$ is defined as the sub-sigma field of $\mathcal F^{\infty}$ composed of invariant events, that is events $A \in \mathcal F^{\infty}$ such that $\theta_t^{-1} A=A$ for all $t >0$.
It is well know that the Poisson boundary identifies with the set of bounded harmonic functions for the generator of the process, see Proposition (3.4) of \cite{revuz}.
\par
\medskip
As announced in Section \ref{sec.intro}, our main tool to determine the Poisson boundary of the relativistic diffusion will be the d\'evissage method introduced by the authors in the paper \cite{devissage}. For the sake of self containess, let us recall here the framework and main results of the latter.
Let $X$ be a differentiable manifold and $G$ a finite dimensional connected Lie group, in particular $G$ carries a right invariant Haar measure $\mu$. 
If $K$ is a compact sub-group of $G$, we will denote by $G/K$ the associated homogenous space and by $\pi : G \to G/K$ the canonical projection.
As usual, let us denote by $C^{\infty}(X \times G, \mathbb R)$ the set of smooth functions from $X \times G$ to $\mathbb R$. From the natural left action of $G$ on itself
\[
\begin{array}{cll}  G \times G & \to & G \\   (g, h) & \mapsto & g.h := gh \end{array},
\]
we deduce a left action of $G$ on $C^{\infty}(X \times G, \mathbb R)$, namely:
\[ 
\begin{array}{cll}  G \times C^{\infty}(X\times G, \mathbb R) & \to & C^{\infty}(X \times G, \mathbb R) \\   (g, f) & \mapsto & g \cdot f := \left( (x,h) \mapsto f(x,g.h) \right)\end{array}.
\]
In this context, let $(X_s, Y_s)_{s \geq 0}$ be a diffusion process with values in $X \times G$ and with infinite lifetime. We denote by $\mathcal L$ its infinitesimal generator acting on $C^{\infty}(X \times G, \mathbb R)$. The law of a sample path $(X_s, Y_s)_{s \geq 0}$ starting from $(x,y)$ will be denoted by $\mathbb P_{(x,y)}$, and $\mathbb E_{(x,y)}$ will denote the associated expectation.
The hypotheses under which the d\'evissage method can be applied are of different nature and are the following.
\begin{hypo}\textcolor{white}{blanc}
\begin{enumerate}
\item the process $(X_s)_{s \geq 0}$ is a subdiffusion of $(X_s,Y_s)_{s \geq 0}$. The sigma field $\textrm{Inv}((X_s)_{s \geq 0})$ is either trivial or generated by a random variable $\ell_{\infty}$ with values in a separable measure space $(S,\mathcal G, \lambda)$, the law of $\ell_{\infty}$ being absolutely continuous with respect to $\lambda$. 
\item for any starting point $(x,y) \in X \times G$, the process $(Y_s)_{s \geq 0}$ converges $\mathbb P_{(x,y)}-$almost surely when $s$ goes to infinity to a random variable $Y_{\infty}$ in $G$.
\item  the infinitesimal generator $\mathcal L$ of the diffusion is equivariant under the action of $G$ on the space $C^{\infty}(X \times G, \mathbb R)$, i.e. $
\mathcal L(g \cdot f) = g \cdot (\mathcal L f)$, $\forall f \in C^{\infty}(X\times G, \mathbb R)$. 
\item all bounded $\mathcal L-$harmonic functions are continuous on $X \times G$.
\end{enumerate}
\end{hypo}

The main result of \cite{devissage} is then the following. 

\begin{theo}[Theorem 1 and 2 of \cite{devissage}]\label{theo.trivial}
Suppose that the full diffusion $(X_s,Y_s)_{s \geq 0}$ satisfies the above d\'evissage conditions, then the two sigma fields  
$\textrm{Inv}((X_s, Y_s)_{s \geq 0}) $ and $\textrm{Inv}((X_s)_{s \geq 0}) \vee \sigma(Y_{\infty})$
coincide up to $\mathbb P_{(x,y)}-$negligeable sets. More precisely 
\begin{enumerate}
\item if $\textrm{Inv}((X_s)_{s \geq 0})$ is trivial, then $\textrm{Inv}((X_s, Y_s)_{s \geq 0}) $ coincides with  $\sigma(Y_{\infty})$ up to negligeable sets. Equivalently, if $H$ is a bounded $\mathcal L-$harmonic function, then there exists a bounded mesurable function $\psi$ on $G$ such that $H(x,y)=\mathbb E_{(x,y)}[\psi(Y_{\infty})]$, for all $(x,y) \in X \times G$.
\item if $\textrm{Inv}((X_s)_{s \geq 0})$ is generated by a random variable $\ell_{\infty} \in S$, then $\textrm{Inv}((X_s, Y_s)_{s \geq 0}) $ coincides with  $\sigma(\ell_{\infty},Y_{\infty})$ up to negligeable sets. Equivalently, if $H$ is a bounded $\mathcal L-$harmonic function, then we have  $H(x,y)=\mathbb E_{(x,y)}[\psi(\ell_{\infty},Y_{\infty})]$ for all $(x,y)\in X\times G$, for a bounded mesurable function $\psi$ on $S\times G$.
\end{enumerate}
\end{theo}

The above results in the Lie group framework can be naturally extended to the case where the Lie group $G$ is replaced by an homogeneous space $G/K$, as soon as the homogeneous diffusion can be lifted to a diffusion satisfying the preceeding hypotheses. 

\begin{theo}[Theorem 3 of \cite{devissage}]\label{theo.homo}
Let $(X_s,Y_s)_{s \geq 0}$ with values in $X \times G/K$, such that there exists a $K$-right equivariant diffusion $(X_s, G_s)_{s \geq 0}$ in $X \times G$ satisfying the above d\'evissage conditions and such that under $\mathbb P_{(x,y)}$ the process  $(X_s, Y_s)_{s \geq 0}$ has the same law as  $(X_s, \pi(G_s))_{s \geq 0}$ under $\mathbb P_{(x,g)}$ for $ g\in \pi^{-1}(\{ y \})$. Then for all starting points $(x,y) \in X \times G/K$, the two sigma fields  
$\textrm{Inv}((X_s, Y_s)_{s \geq 0}) $ and $ \textrm{Inv}((X_s)_{s \geq 0}) \vee \sigma(Y_{\infty})$ 
coincide up to $\Prob_{(x,y)}-$negligeable sets. 
\end{theo}

\newpage
\section{Statement of the results} \label{sec.statement}
In this section, we state the main results of the article, namely we determine the long time asymptotics and the Poisson boundary of the relativistic diffusion in all Lorentzian model space-times and in expanding Robertson--Walker space-times. This allows us to compare the stochastic compactification of the underlying manifolds given by the exit points of the relativistic diffusion, to the other (purely) geometric compactifications such as the conformal or causal boundaries.

\subsection{Lorentz model manifolds}

\label{sec.statementmodel}

Let us first consider the case of model space-times, i.e. Lorentzian manifolds with constant scalar curvature. The only two cases where the Poisson boundary of the relativistic Brownian motion was previously fully determined are the causal model space-times i.e. Minkowski and de Sitter space-times. In order to give a complete picture of what happen on model space-times, let us first recall these results.
\par
\medskip
As already noticed above, in causal model space-times, the conformal and causal boundaries coincide, namely they both topologically identify with a cone $\mathbb R^+ \times \mathbb S^{d-1}$ in the case of Minkowski space-time, and with a sphere $\mathbb S^{d}$ in the case of de Sitter space-time.
Note that in a causal Lorentzian manifold, by definition, any inextensible causal path converges to a point of the causal boundary so that it is obvious that the relativistic diffusion in the case of $\mathcal M= \R^{1,d}$ or $\dS$ will converge to a random point $\xi_{\infty} \in \partial \mathcal M$. What is remarkable here is that there is no extra invariant information, i.e. the Poisson boundary is fully described by the single random variable $\xi_{\infty}$.
\par
\medskip
In the case of Minkowski space-time this result was first established by Bailleul in \cite{ismael}, by Bailleul and Raugi in \cite{bailleul_raugi} and was also recovered by the authors in \cite{devissage} using the d\'evissage method.

\begin{theo}[\cite{ismael}, \cite{bailleul_raugi} and \cite{devissage}] \label{theo.poisson.Minko}
Let $(\xi_0, \dot{\xi}_0) \in T^1_+ \mathcal \R^{1,d}$ and let $(\xi_s, \dot{\xi}_s)_{s \geq 0}$ be the relativistic diffusion $T^1_+\mathcal \R^{1,d}$ starting from $(\xi_0, \dot{\xi}_0)$. 
Then almost surely as $s$ goes to infinity, the process $(\xi_s)_{s \geq 0}$ with values in $\R^{1,d}$ converges to a random point $\xi_{\infty}$ of the causal boundary $\partial \R^{1,d}$.  Moreover, the invariant sigma field of the full process $\text{Inv}((\xi_s, \dot{\xi}_s)_{s \geq 0})$ coincides with $\sigma(\xi_{\infty})$ up to negligeable sets.
\end{theo}
Figure \ref{fig.1} below illustrates the convergence of the process $(\xi_s)_{s\geq 0}$ to a random point $\xi_{\infty}$ of the causal / conformal boundary, which topologically identifies here with a cone $\mathbb R^+ \times \mathbb S^{d-1}$. A point $\xi_{\infty}=(\delta_{\infty},\theta_{\infty}) \in \partial \mathbb R^{1,d} \approx \mathbb R^+ \times \mathbb S^{d-1}$ corresponds to causal curves which go to infinity in the direction $\theta_{\infty} \in \mathbb S^{d-1}$ along an affine hyperplane characterized by the scalar $\delta_{\infty} \geq 0$. Let us notice that, tranversally to this hyperplane,  the process has non-converging fluctuations and refer to \cite{ismael} or Section 4.2.1 of \cite{devissage} for more details. 
\begin{figure}[ht]
\hspace{2cm}\begin{minipage}{14cm}
\includegraphics[scale=0.5]{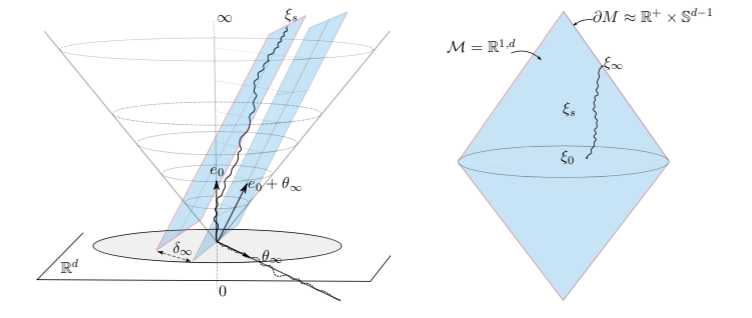}
\end{minipage}
\caption{Asymptotics of the relativistic diffusion in Minkowski space-time $\mathbb R^{1,d}$  and in the associated Penrose diagram.}\label{fig.1}
\end{figure}

The identification of the Poisson boundary of the relativistic diffusion with values in de Sitter space-time $\mathrm{dS}^{d+1}$ was established by the first author in \cite{angst1}.  We propose here in Section \ref{sec.ads} below a new /  alternative proof exploiting the Lie group and warped product structures of the unitary tangent bundle. 
\begin{theo}[ \cite{angst1}] \label{theo.poisson.dS}
Let $(\xi_0, \dot{\xi}_0) \in T^1_+ \dS$ and let $(\xi_s, \dot{\xi}_s)_{s \geq 0}$ be the relativistic diffusion $T^1_+\dS$ starting from $(\xi_0, \dot{\xi}_0)$. 
Then almost surely as $s$ goes to infinity, the process $(\xi_s)_{s \geq 0}$ with values in $\dS$ converges to a random point $\xi_{\infty}$ of the causal boundary $\partial \dS$.  Moreover, the invariant sigma field of the full process $\text{Inv}((\xi_s, \dot{\xi}_s)_{s \geq 0})$ coincides with $\sigma(\xi_{\infty})$ up to negligeable sets.
\end{theo}
Via a stereographic projection, Figure \ref{fig.2} below represents the de Sitter space-time in the projective space as the complementary set of a ball. The causal boundary then identifies with the topological boundary i.e the sphere $\mathbb S^d$, and the typical almost sure bahavior of the relativistic diffusion is to go to infinity towards a random point on this sphere. 
\begin{figure}[ht]
\begin{center}
\includegraphics[scale=0.5]{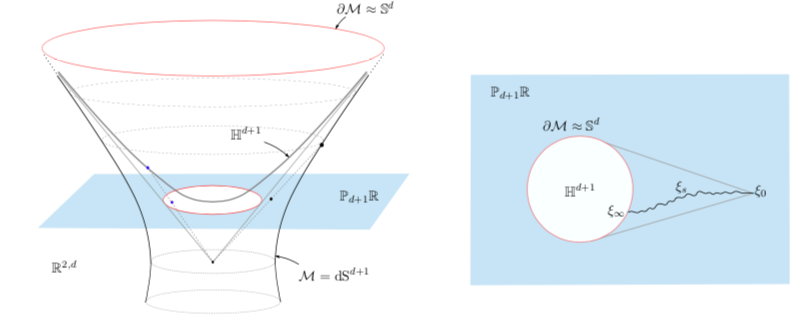}
\end{center}
\caption{Projective image of de Sitter space-time and asymptotics of the relativistic diffusion.}\label{fig.2}
\end{figure}

\begin{rema}\label{rem.model}
Let us emphasize here that in the two cases of causal model space-times $\R^{1,d}$ and $\dS$, the probabilistic information at infinity of the full relativistic diffusion $(\xi_s, \dot{\xi}_s)$ is thus carried by the single first projection $\xi_s$ with values in $\mathcal M$. In particular, no information is hidden in the asymptotic behavior of the derivative $\dot{\xi}_s$, nor in its interaction with its antiderivative. 
Moreover, the set of exit points of the diffusion, which can be seen as a stochastic compactification of the base manifold, actually coincides with the purely geometric (causal or conformal) boundaries.
\end{rema}

Let us now turn to new results and consider the last case of model space-time, namely the Anti de Sitter space-time. As already noticed in Section \ref{sec.background}, this space-time is not causal and therefore has no causal boundary. Nevertheless, it has a natural conformal compactification which identifies with the Einstein static space-time $\Ein$. The following theorem is proved in Section \ref{sec.ads} below, exploiting again the Lie group and warped product structures of the unitary tangent bundle. 

\begin{theo} \label{theo.poisson.AdS}
Let  $(\xi_0, \dot{\xi}_0) \in T^1_+ \AdS$ and let $(\xi_s, \dot{\xi}_s)_{s \geq 0}$ be a relativistic diffusion process in $T^1_+\AdS$ starting from $(\xi_0, \dot{\xi}_0)$. 
Then almost surely as $s$ goes to infinity, the process $(\xi_s)_{s \geq 0}$ is asymptotic to a random light circle $\ell_{\infty}$ in the conformal boundary identified with the Einstein conformal manifold $\Ein$. The sample path $(\xi_s, \dot{\xi}_s)_{s \geq 0}$ carries an extra invariant information given by a random point $p_{\infty} \in \ell_{\infty}$ on the limit light circle and the invariant $\sigma$-field of the full diffusion $(\xi_s, \dot{\xi}_s)_{s\geq0}$ coincides with $\sigma(p_{\infty},\ell_{\infty})$ up to negligeable sets. 
\end{theo}

In Figure \ref{fig.AdShyper} below the Anti de Sitter space-time is seen in the projective space as the interior of an one-sheeted hyperboloid which represent its conformal boundary (the Einstein universe of dimension $d$). A path of the relativistic diffusion is represented being asymptotic to a light circle in the hyperboloid (a line in the projective space so a circle topologically).

\begin{figure}[ht]
\begin{center}
\includegraphics[scale=0.5]{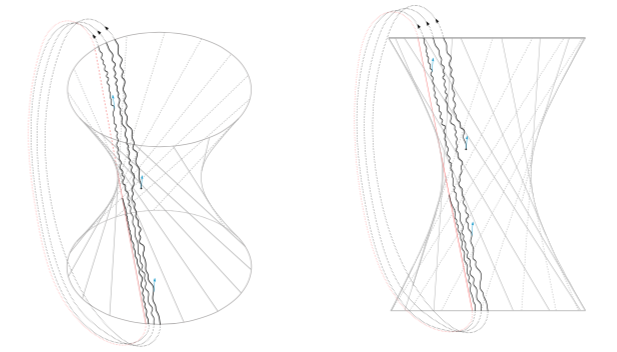}
\end{center}
\caption{Asymptotics in the projective  image of Anti de Sitter space-time.}\label{fig.AdShyper}
\end{figure}

More (geometric) details on the almost sure behavior of the relativistic diffusion on de Sitter and Anti de Sitter space-times are given in Section  \ref{sec.geomeaning} below. In particular, the significance of the above ``extra information'', i.e. the fact that random limit light circle is pointed is explained. 
\begin{rema}
The above Theorem \ref{theo.poisson.AdS} is thus the first example where the support of the Poisson boundary of the relativistic Brownian motion does not coincide with the natural geometric boundary. More precisely, the set of exit points of the relativistic diffusion is here richer than the conformal boundary. We will see in the next section that this fact, which could a priori sounds like an exception in the constantly curved case, is actually the rule in curved space-times.
\end{rema}

\subsection{Robertson--Walker space-times}

Let us now describe, in an exhaustive way, the asymptotic behavior of the relativistic diffusion in expanding Robertson--Walker space-times. Let us recall that the warping function $\alpha$ is assumed to $\log-$concave, in other words its logarithmic derivative $H:=\alpha'/\alpha$, known as the Hubble function, is assumed to be decreasing. From a geometric point of view, the geometry at (the future oriented) infinity of the manifold $\mathcal M=(0, +\infty) \times_{\alpha} M$ differs drastically depending on the finiteness of the integral 
\[
I(\alpha):=\int^{\infty} \frac{du}{\alpha(u)}.
\]
In the case where the integral $I(\alpha)$ is finite, the manifold is said to have a finite horizon, meaning in particular that the projections of light-like geodesics in the Riemannian fiber $M$ are convergent. Moreover, the (future oriented part of the) causal boundary of the manifold then identifies with a space-like copy of $M$, see \cite{flores}. The sample paths of the relativistic diffusion then obbey the following dichotomy, depending on the integrability of the Hubble function.

\begin{theo}\label{theo.finitehori}
Let $\mathcal M:=(0, +\infty) \times_{\alpha} M$ be a Robertson--Walker space-time such that $I(\alpha) < +\infty$.
Let $(\xi_0, \dot{\xi}_0) \in T^1_+ \mathcal M$ and let  $(\xi_s, \dot{\xi}_s)_{s \geq 0}=(t_s,x_s, \dot{t}_s, \dot{x}_s)_{s \geq 0}$ be the relativistic diffusion in $T^1_+\mathcal M$ starting from $(\xi_0, \dot{\xi}_0)$. 
\begin{enumerate}
\item if  $H^3 \notin \mathbb L^1$, then almost surely as $s$ goes to infinity, the process $(x_s)_{s \geq 0}$ converges to a random point $x_{\infty}$ in $M$, and the invariant sigma field of the full diffusion $\text{Inv}((\xi_s, \dot{\xi}_s)_{s \geq 0})$ coincides with $\sigma(x_{\infty})$ up to negligeable sets. 
\item if  $H^3 \in \mathbb L^1$, then almost surely as $s$ goes to infinity, the process $Y_s:=(x_s, \dot{x}_s / |\dot{x}_s|)_{s \geq 0}$ converges to a random point $Y_{\infty}$ in $T^1 M$, and the invariant sigma field of the full diffusion $\text{Inv}((\xi_s, \dot{\xi}_s)_{s \geq 0})$ coincides with $\sigma(Y_{\infty})$ up to negligeable sets. 
\end{enumerate}
\end{theo}

\begin{figure}[ht]
\begin{center}
\includegraphics[scale=0.5]{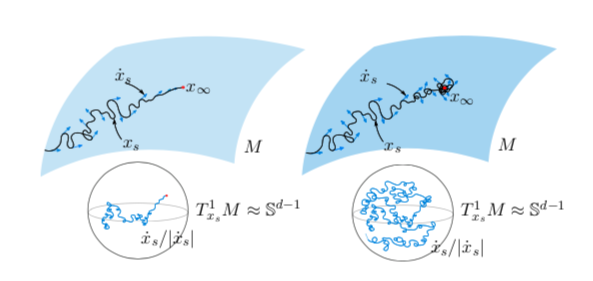}
\end{center}
\caption{Asymptotics of the projection of the relativistic diffusion in $T^1 M$ with finite horizon and depending on the fact that $H^3$ is integrable at infinity (left) or not (right).}
\end{figure}

\begin{rema}
In the first case where $H^3 \notin \mathbb L^1$, the support of the Poisson boundary of the relativistic diffusion thus again identifies with the purely geometric causal boundary $\partial \mathcal M$, i.e. the Riemannian fiber $M$ here. 
This happens for example for torsion functions $\alpha$ with exponential or sub-exponential growth, e.g. of the form $\alpha(t) = e^t$ or $\alpha(t)=e^{t^\beta}$ with $2/3<\beta<1$.
Nevertheless, in the second case where $H^3 \in \mathbb L^1$, which happens for example for all the torsion functions $\alpha$ with polynomial growth, the normalized derivative also converges to a random variable and the set of exit points of the diffusion is then richer than the causal boundary. 
\end{rema}

We now turn to the case of Robertson--Walker space-times $\mathcal M=(0, +\infty) \times_{\alpha} M$ with infinite horizon, i.e. $I(\alpha)=+\infty$. In this case, the asymptotic behavior of the diffusion heavily depends on the geometry of the Riemannian fiber $M=\mathbb R^d, \mathbb S^d$ or $\mathbb H^d$, this is the reason why we state our results separately.
Let us begin with the flat fiber case $M=\mathbb R^d$ and for simplicity, let us assume here (and here only) that the torsion function $\alpha$ has polynomial growth in the sense that there exists $0<c \leq 1$ such that $\lim_{t \to +\infty} H(t) \times t = c$.
\begin{theo}\label{theo.infhorEucli}
Let $\mathcal M:=(0, +\infty) \times_{\alpha} \mathbb R^d$ be a Robertson--Walker space-time such that $\alpha$ has polynomial growth and $I(\alpha) = +\infty$. Let $(\xi_0, \dot{\xi}_0) \in T^1_+ \mathcal M$ and let  $(\xi_s, \dot{\xi}_s)_{s \geq 0}=(t_s,x_s, \dot{t}_s, \dot{x}_s)_{s \geq 0}$ be the relativistic diffusion in $T^1_+\mathcal M$ starting from $(\xi_0, \dot{\xi}_0)$. 
Then almost surely as $s$ goes to infinity, the two processes $(\Theta_s)_{s \geq 0}:=(\dot{x}_s / |\dot{x}_s|)_{s \geq 0}$ and $(\delta_s)_{s \geq 0}:=(x_s - \Theta_s \int_1^{t_s} du/\alpha(u))_{s \geq 0}$
converge to random points $\Theta_{\infty} \in \mathbb S^{d-1}$ and $\delta_{\infty} \in \mathbb R^d$ respectively, and the invariant sigma field of the full diffusion $\text{Inv}((\xi_s, \dot{\xi}_s)_{s \geq 0})$ coincides with $\sigma(\Theta_{\infty},\delta_{\infty})$ up to negligeable sets. 
\end{theo}

\begin{figure}[ht]
\begin{center}
\includegraphics[scale=0.45]{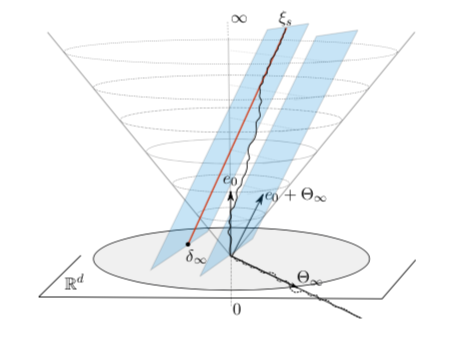}
\end{center}
\caption{Asymptotics of the relativistic diffusion under the hypotheses of Theorem \ref{theo.infhorEucli}.}\label{fig.3}
\end{figure}

\begin{rema}Compared to Theorem \ref{theo.poisson.Minko} illustrated in Figure \ref{fig.1} and where the relativistic diffusion is asymptotic to a random hyperplane with non-converging transversal fluctuations, under the hypotheses of Theorem  \ref{theo.infhorEucli}, it asymptotically describes a random line, more precisely a light-like geodesics encoded by the direction $\Theta_{\infty}$ and a point $\delta_{\infty} \in \mathbb R^d$. For a Robertson--Walker space-time of the form $\mathcal M=(0+\infty) \times_{\alpha} \mathbb R^d$ with $I(\alpha)=+\infty$, the causal boundary $\partial M$ coincides with the one of Minkowski space-time described above, namely it is still topologically a cone $\mathbb R^+ \times \mathbb S^{d-1}$ which can be interpreted as a set of hyperplane, not a set a lines. Therefore, the stochastic compactification of the base manifold given by the set of exit points of the relativistic diffusion is again richer than its purely geometric analogues.
\end{rema}

In the spherical case $M=\mathbb S^d$, and in infinite horizon i.e. $I(\alpha)=+\infty$, the causal boundary of the warped product $\mathcal M:=(0, +\infty) \times_{\alpha}  \mathbb S^d$ is reduced to a point, see \cite{flores}. Nevertheless, the Poisson boundary of the relativistic diffusion is non-trivial since the spherical projection of the process asymptotically describe a random big circle in $\mathbb S^d$. 

\begin{theo}\label{theo.infhorSpher}
Let $\mathcal M:=(0, +\infty) \times_{\alpha}  \mathbb S^d$ be a Robertson--Walker space-time such that $I(\alpha) = +\infty$. Let $(\xi_0, \dot{\xi}_0) \in T^1_+ \mathcal M$ and let  $(\xi_s, \dot{\xi}_s)_{s \geq 0}=(t_s,x_s, \dot{t}_s, \dot{x}_s)_{s \geq 0}$ be the relativistic diffusion in $T^1_+\mathcal M$ starting from $(\xi_0, \dot{\xi}_0)$. 
Then almost surely as $s$ goes to infinity, the process $(x_s)_{s \geq 0}$ asymptotically describes a random big circle in $\mathbb S^d$ generated by two orthogonal random vectors $(u_{\infty}, v_{\infty}) \in \mathbb S^d \times \mathbb S^d$  and the invariant sigma field of the full diffusion $\text{Inv}((\xi_s, \dot{\xi}_s)_{s \geq 0})$ coincides with $\sigma(u_{\infty},v_{\infty})$ up to negligeable sets. 
\end{theo}

\begin{figure}[ht]
\begin{center}
\includegraphics[scale=0.5]{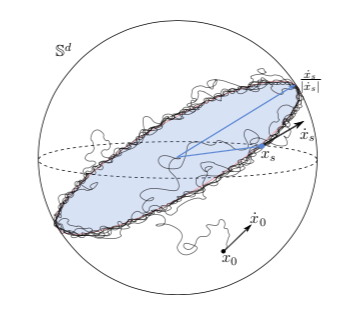}
\end{center}\label{fig.sphere}
\caption{Asymptotics of the relativistic diffusion in Robertson--Walker space-times with infinite horizon and spherical fiber.}
\end{figure}

\noindent
We conclude with the negatively curved case $M=\mathbb H^d$. 
\begin{theo}\label{theo.infhorHyp}
Let $\mathcal M:=(0, +\infty) \times_{\alpha} \mathbb H^d$ be a Robertson--Walker space-time such that $I(\alpha) = +\infty$. Let $(\xi_0, \dot{\xi}_0) \in T^1_+ \mathcal M$ and let  $(\xi_s, \dot{\xi}_s)_{s \geq 0}=(t_s,x_s, \dot{t}_s, \dot{x}_s)_{s \geq 0}$ be the relativistic diffusion in $T^1_+\mathcal M$ starting from $(\xi_0, \dot{\xi}_0)$. 
Then almost surely as $s$ goes to infinity, the process $(x_s)_{s \geq 0}$ goes to infinity along a random direction $\theta_{\infty} \in \mathbb S^{d-1} \sim \partial \mathbb H^d$. Moreover, the scalar process $(\delta_s)_{s \geq 0}$ defined by $\delta_s := \sinh^{-1}(|x_s|) - \int_1^{t_s} du /\alpha(u)$ converges almost surely to a random point $\delta_{\infty} \in \mathbb R$. The invariant sigma field of the full diffusion $\text{Inv}((\xi_s, \dot{\xi}_s)_{s \geq 0})$ coincides with $\sigma(\theta_{\infty},\delta_{\infty})$ up to negligeable sets. 
\end{theo}

\begin{rema}
In this last case of a Robertson--Walker space-time with infinite horizon and hyperbolic fiber, the (future oriented part of the) causal boundary also identifies topologically with a cone of the type $\mathbb R^+ \times \mathbb S^{d-1}$, see \cite{flores}. Making the parallel with the case of a Euclidean fiber, the relativistic diffusion goes here to infinity in a the random direction $\theta_{\infty}$ along a random hypersurface (not an hyperplane as before) characterized by the scalar $\delta_{\infty} \geq 0$, see Figure  \ref{fig.hyper} below. As in the flat Minkowskian case, the support of the Poisson boundary thus coincides here again with the natural geometric compactification of the base manifold.
\end{rema}

\begin{figure}[ht]
\begin{center}
\includegraphics[scale=0.5]{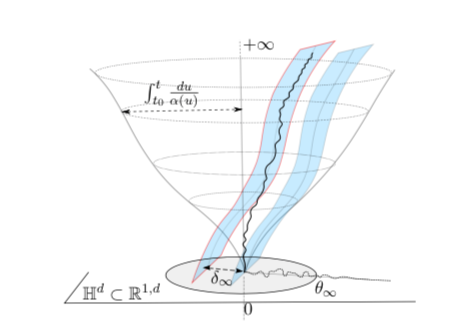}
\end{center}
\caption{Asymptotics of the relativistic diffusion in Robertson--Walker space-times with infinite horizon and hyperbolic fiber.}\label{fig.hyper}
\end{figure}

\newpage
\section{Proof of the results}\label{sec.proof}
We now give the proofs of the results stated in the last section. The next Section \ref{sec.model} is devoted to the proof of Theorem \ref{theo.poisson.dS} and \ref{theo.poisson.AdS} concerning the Poisson boundary of the relativistic diffusion in model space-times. The case of Robertson--Walker space-times is then treated in Section \ref{sec.RWglobal}.

\subsection{Poisson boundary in model manifolds}\label{sec.model} 

\label{sec.ads}

As mentioned in Section \ref{sec.statementmodel}, using elaborated coupling techniques, Bailleul first established Theorem \ref{theo.poisson.Minko} in \cite{ismael}, i.e. the fact that the support of the Poisson boundary of the relativistic diffusion in Minkowski space-time identifies with the causal / conformal boundary. Then using techniques from random walks on Lie groups, this result was recovered by Bailleul and Raugi  in \cite{bailleul_raugi}. As a first application of the d\'evissage method, we also gave a concise proof of this result in \cite{devissage}. The case of de Sitter space-time, i.e. Theorem \ref{theo.poisson.dS} above, appears as a particular case of the main result of \cite{angst1}. Indeed, it is well known that de Sitter space-time can be globally written as a Robertson--Walker space-time with a flat fiber and with  a torsion function with exponential growth, which is precisely the framework of the latter article. 
\par
\medskip
We are thus left with the proof of Theorem \ref{theo.poisson.AdS} concerning the Anti de Sitter space-time. The method we follow here will in fact provide a proof of both Theorems  \ref{theo.poisson.dS} and  \ref{theo.poisson.AdS} at the same time.
As the sphere and the hyperbolic space, the de Sitter $\dS$ and Anti de Sitter $\AdS$ space-times are homogenous spaces and their orthonormal frame bundle are identified with Lie groups which are respectively 
$\mathrm{PSO}(1,d+1)$ and $\mathrm{PSO}(2,d)$. The relativistic diffusion can naturally be lifted to a diffusion with values in the orthonormal frame bundle. Those Lie groups are semi-simple and the study of the asymptotic behavior of the (lifted) relativistic diffusion can be done using well known results on random walks on semi-simple Lie groups. The plan of the proof of Theorems \ref{theo.poisson.dS} and  \ref{theo.poisson.AdS} is then the following: in the next Section \ref{sec.lift}, we precise the geometric setting and the generator of the lifted diffusion process ; then in Section \ref{sec.iwa}, using suitable Iwasawa coordinates, we determine its almost sure asymptotics using standard results for Markov processes on Lie groups ; the geometric meaning of these convergence results is detailed in Section \ref{sec.geomeaning}. Finally, in Section \ref{sec.poisson.model}, we compute the Poisson boundary of the process using the d\'evissage method. 

\subsubsection{Lifting the relativistic diffusion} \label{sec.lift}

As recalled in section \ref{sec.background}, the de Sitter space-time $\dS$ is the unit sphere in $\R^{1,d+1}$ of the quadratic form $Q_{1,d+1}$. Let denote by $\mathrm{PSO}(1,d+1)$ the connected component of the identity of $\mathrm{SO}(1,d+1)$, the group of $Q_{1,d+1}$-isometries of determinant $1$. Then $\mathrm{PSO}(1,d+1)$ acts by isometry and transitively on $\dS$ which can be identified with the homogenous space $\mathrm{PSO}(1,d+1)/\mathrm{SO}(1,d)$, $\mathrm{SO}(1,d)$ being identified with the stabilizer of one point in $\dS$. Denote by $(e_0, e_1, \dots, e_{d+1})$ the canonical basis of $\R^{1, d+1}$, and so $Q_{1,d+1}(e_0)=-1$ and $Q_{1,d+1}(e_i)=1$ for $i \in \{1, \dots, d+1\}$. Define by $\tilde{\pi}$ the projection 
\[
\begin{matrix}
\tilde{\pi}: & \mathrm{PSO}(1,d+1) & \longrightarrow & \dS \\
& g & \longmapsto & g(e_1)
\end{matrix}
\]
i.e. the orthonormal frame bundle projection onto $\dS$. An element $g$ of $\mathrm{PSO}(1,d+1)$ provides $g(e_1) \in \dS$ and also an orthonormal basis $(g(e_0), g(e_2), \dots, g(e_{d+1}))$ of $T_{g(e_1)} \dS$.
Since $\mathrm{PSO}(1,d+1)$ acts transitively on the unitary tangent bundle $T^1_{+} \dS$, it can be  identified with the quotient $\mathrm{PSO}(1,d+1)/ \mathrm{SO}(d)$ via the projection $\pi$ defined by 
\[
\begin{matrix}
\pi : & \mathrm{PSO}(1,d+1) & \longrightarrow & T^{1}_{+} \dS \\
& g & \longmapsto & (g(e_0), g(e_1) ),
\end{matrix}
\]
the stabilizer of $(e_0,e_1)$ being isomorphic to $\mathrm{SO}(d)$. 
In the same manner, recall that the Anti de Sitter space-time $\AdS$ is a unit sphere in $\R^{2,d}$ for the quadratic form $Q_{2,d}$
\[
\AdS = \{ \xi \in \R^{d+1}, \quad Q_{2,d}(\xi)=-1 \},
\]
identified with the homogenous space $\mathrm{PSO}(2,d)/\mathrm{SO}(d)$ via the projection (denoted $\tilde{\pi}$ again) 
\[
\begin{matrix}
\tilde{\pi}: & \mathrm{PSO}(2,d) & \longrightarrow & \AdS \\
&g & \longmapsto & g(e_0).
\end{matrix}
\]
An element $g\in \mathrm{PSO}(2,d)$ then provides a point $g(e_0)$ in $\AdS$ 	as well as an orthonormal basis $(g(e_1), \dots, g(e_{d+1}))$ of $T_{g(e_0)}\AdS$. Moreover, the unitary tangent bundle $T^{1}_{+} \AdS$ is identified with the homogenous space $\mathrm{PSO}(2,d)/\mathrm{SO}(d)$ via the projection (also denoted by $\pi$)
\[
\begin{matrix}
\pi : & \mathrm{PSO}(2,d) & \longrightarrow & T^{1}_{+} \AdS \\
& g & \longmapsto & (g(e_0), g(e_1) ).
\end{matrix}
\]

In the two cases, $\dS$ and $\AdS$, the relativistic diffusion $(\xi_s, \dot{\xi}_s)_{s \geq 0}$ is thus the projection onto respectively $T^1_{+}\dS$ and $T^1_{+}\AdS$ of a left invariant diffusion $(g_s)_{s \geq 0}$ on a Lie group, respectively $\mathrm{PSO}(1,d+1)$ and $\mathrm{PSO}(2,d)$ whose generator $\tilde{\Op}$ takes the H\"ormander form
\[
\tilde{\Op}= \frac{\sigma^2}{2} \sum_{i=1}^d V_i^2 + H_0,
\]
$(V_i)_{i=1,\dots,d}$ and $H_0$ are left invariant vector fields (respectively vertical and horizontal relatively to the bundle projection $\tilde{\pi}$). Precisely, in the case of $\dS$ 
\[ 
H_0(g) = g(e_0 e_1^{*} + e_1 e_{0}^*), \quad V_i(g)= g (e_{i+1}e_0^{*} + e_0 e_{i+1}^*), \quad 1\leq i \leq d,
\]
and in the case of $\AdS$
\[
H_0(g)=g(e_1e_0^{*}-e_0e_{1}^{*}), \quad V_i(g)= g(e_{i+1}e_1^{*} + e_1 e_{i+1}^*), \quad 1\leq i \leq d.
\]
Saying that $(\dot{\xi}_s, \xi_s)_{t \geq0}$ is the projection by $\pi$ of $(g_s)_{s\geq0}$ is equivalent to the fact that, for every smooth function $f: T^1_{+}\mathcal{M} \to \R$ ($\mathcal{M}=\dS$ or $\AdS$) one has
\[
\tilde{\Op}(f\circ \pi) = \Op(f)\circ \pi.
\]
The asymptotic behavior and the Poisson boundary of the relativistic diffusion $(\xi_s, \dot{\xi}_s)_{s \geq 0}$ will be obtained from the one of $(g_s)_{s \geq 0}$. Note that both $\mathrm{PSO}(1,d+1)$ and $\mathrm{PSO}(2,d)$ are semi-simple Lie groups and moreover the left invariant diffusion $(g_s)_{s \geq 0}$ is hypoelliptic (satisfying H\"ormander condition). Thus, the asymptotic behavior of $(g_s)_{s \geq 0}$ can be made explicit using classical results on random walks and Markov processes (\cite{harry1}, \cite{vircer} \cite{raugi77}, \cite{liao}) on semi-simple Lie groups.

\subsubsection{Asymptotics of the lifted diffusion in Iwasawa coordinates} \label{sec.iwa}
Let us for the moment denote by $G$ for $\mathrm{PSO}(1,d+1)$ or $\mathrm{PSO}(2,d)$ and choose an Iwasawa decomposition $G=NAK$ (and $\mathrm{Lie}(G)= \mathcal{N}\oplus \mathcal{A} \oplus \mathcal{K}$).  Recall that $K$ is a maximal compact subgroup (isomorphic to $\mathrm{SO}(d+1)$ in the case $\mathrm{PSO}(1,d+1)$, and to $(\mathrm{O}(2) \times \mathrm{O}(d) )\cap \mathrm{SL}(d+2)$ in the case of $\mathrm{PSO}(2,d)$), $A$ is an abelian Lie group (of dimension $1$ in the case $\mathrm{PSO}(1,d+1)$ or $2$ in the case $\mathrm{PSO}(2,d)$) and $N$ has a nilpotent Lie algebra $\mathcal{N}$. Then decomposing $g_s$ in Iwasawa coordinates $g_s=n_s a_s k_s $ one observes that $(k_s)_{s \geq 0}$ and $(a_s, k_s)_{s \geq 0}$ are diffusions respectively on $K$ and $A \times K$. Precisely, writing explicitly the dynamics of the diffusion in Iwasawa coordinates, see e.g. \cite{liao}, if $\{\cdot \}_{\mathcal{K}}, \{\cdot \}_{\mathcal{A}}, \{\cdot \}_{\mathcal{N}}$ denotes the Iwasawa projections of $\mathrm{Lie}(G)$ onto $\mathcal{K},\mathcal{A},\mathcal{N}$, we get
\begin{eqnarray}
dk_s &=&  \displaystyle{\sigma \sum_{i=1}^d \{ \mathrm{Ad}(k_s) V_i  \}_{\mathcal{K}} k_s \circ dB^{i}_{s} + \{ \mathrm{Ad}(k_s) H_0 \}_{\mathcal{K}} k_s ds}, \label{eq.iwa.k} \\  \nonumber
\\ 
da_s &=& \displaystyle{ \sigma \sum_{i=1}^d a_s \{  \mathrm{Ad}( k_s) V_i \}_{\mathcal{A}}  \circ dB_{s}^{i} + a_s \{ \mathrm{Ad}(k_s) H_0 \}_{\mathcal{A}} ds, } \label{eq.iwa.a}\\  \nonumber
\\
dn_s &=&  \displaystyle{\sigma \sum_{i=1}^d n_s \mathrm{Ad}(a_s) \{ \mathrm{Ad}(k_s) V_i  \}_{\mathcal{N}} \circ dB_s^{i} + n_s \mathrm{Ad}(a_s) \{ \mathrm{Ad}(k_s) H_0 \}_{\mathcal{N}} ds.}\label{eq.iwa.n}
\end{eqnarray}
In particular, from Equations \eqref{eq.iwa.k}, \eqref{eq.iwa.a}, \eqref{eq.iwa.n}, one can easily check that the processes $(k_s)_{s \geq 0}$ and $(k_s,a_s)_{s \geq 0}$ are subdiffusions of the full process $(k_s,a_s,n_s)_{s \geq 0}$.
In our particular cases, since the projection of $(g_s)_{s\geq0}$ onto the unitary fiber bundle identified with $G/\mathrm{SO}(d)$ is a diffusion, it follows that the projection $(\theta_s)_{s \geq 0}$ of $(k_s)_{s\geq 0}$ onto $K/\mathrm{SO}(d)$ (identified with  $\mathbb{S}^{d}$ for $\dS$ and $\mathbb{S}^1$ for $\AdS$) is itself a diffusion process, i.e. a Markov process, whose dynamics is explicit. 
Precisely in the Anti de Sitter case, using the explicit Iwasawa decomposition of $\mathrm{PSO}(2,d)$ detailed in the next Section \ref{sec.geomeaning} and denoting by $\theta_s \in \mathbb{S}^{1}$ the element of the circle defined by the relations $k_s(e_0)= \cos(\theta_s) e_0 + \sin(\theta_s) e_1$ and $k_s(e_1)= -\sin(\theta_s) e_0 + \cos(\theta_s) e_1$, the SDE solved by $(k_s)_{s\geq0}$ provides the following explicit dynamics of $(\theta_s)_{s\geq 0}$ 
\begin{equation}\label{eq.theta1}
d\theta_s = \sigma \cos(\theta_s) dW_s + \left ( 1 + \frac{\sigma^2 (d-2)}{4} \sin(2 \theta_s)   \right ) ds,
\end{equation}
where $(W_s)_{s\geq0}$ is an usual Brownian motion. 
In the case of de Sitter space-time, considering the explicit Iwasawa decomposition of $\mathrm{PSO}(1,d+1)$ also given in Section \ref{sec.geomeaning} below and denoting $\theta_s := k_s(e_1) \in \mathbb{S}^{d}$, one gets the following dynamics of $(\theta_s)_{s\geq0}$ obtained from the SDE solved by $(k_s)_{s\geq0}$,
\begin{equation}\label{eq.theta2}
d \theta_s = -\sigma (e_1^{*} \theta_s) dM^{\theta}_s - \frac{\sigma^2}{2} \big [ (d-2)(e_1^* \theta_s)e_1 + 2 (e_1^* \theta_s)^2 \theta_s \big ]ds+ (e_1 \theta_s^* -\theta_s e_1^* )\theta_s ds,
\end{equation}
where the martingale term is $dM^{\theta}_s := \sum_{i=1}^d k_s(e_{i+1}) dB^{i}_s$ and has covariation bracket given by $(\mathrm{Id}-\theta_s^{*}\theta_s)ds$. 
Using classical results on finite dimensional diffusion processes, one can then easily get that $(\theta_s)_{s\geq 0}$ admits a unique invariant probability measure on $K/\mathrm{SO}(d)$ and is hence ergodic. In particular, the Poisson boundary $\mathrm{Inv}((\theta_s)_{s \geq 0})$ is trivial.
\par
\medskip
Let us now concentrate on the diffusion process $(k_s, a_s)_{s\geq 0}$. In the de Sitter case, the process $(a_s)_{s \geq 0}$ takes values in $\mathbb R$. Starting from Equation \eqref{eq.iwa.a} and using It\^o formula, the Doob--Meyer decomposition of $\log a_s$ is the sum of a martingale term and an additive functional of $(\theta_s)_{s \geq0}$. By the ergodic theorem, one then deduces that $\frac{1}{s} \log a_s$ converges almost-surely as $s$ goes to infinity to some deterministic constant $\eta_{\infty} \in \mathbb R$. In the same manner, in the Anti de Sitter case,  it is easily seen that the matrix logarithm $\frac{1}{s} \log a_s$ converges almost surely to some deterministic limit  $\eta_{\infty}$.
A fundamental and well known fact is that, in both cases, the limit $\eta_{\infty}$ then necessarily belongs to the interior of the Weyl chamber corresponding to the Iwasawa decomposition chosen. Indeed, this result is classical for random walks generated by a density measure, see e.g. the standard references \cite{harry1},\cite{vircer},\cite{guivarc'h_raugi} or \cite{liao}. Although we are working here with diffusion processes, the study can naturally be reduced to the ``random walk'' framework by discretizing the time, with of course no effect on the fact that the limit $\eta_{\infty}$ is in the interior of the Weyl chamber.

\begin{rema}
In the case of de Sitter space-time, due to the symmetry of Equation \eqref{eq.theta2}, the invariant measure of $(\theta_s)$ is explicit and the positivity of the limit $\eta_{\infty}=\lim_{s \to +\infty} \frac{1}{s} \log a_s$ can be checked easily by an explicit and elementary computation.  In the Anti de Sitter case however, although the invariant measure of $(\theta_s)$  can also be made explicit starting from Equation \eqref{eq.theta1}, computations are much more involved, so that the fact that the limit $\eta_{\infty}$ belongs to the interior of the Weyl chamber seems hard to check using only elementary methods.
\end{rema}

Recall that, by definition, an element $a\in A^{+}$ of the Weyl chamber contracts, by adjonction, an element $Y$ of the nilpotent algebra $\mathcal{N}$, namely
\[
\mathrm{Ad}(a)(Y)= \sum_{\phi \in \Sigma_{+}} e^{-\phi(\log a)} Y_{\phi},
\] 
where $\Sigma_{+}$ is the set of positive roots. From this contraction property and from the fact that $\frac{1}{s}\log a_s$ converges almost surely to an element $\eta_{\infty}$ of the Weyl chamber, one then deduces that the dynamics of $n_s$ is contracting exponentially fast and so that $n_s$ converges almost surely to an asymptotic random variable $n_{\infty}$ with values in $N$. 
The next Proposition sums up the almost sure long time asymptotic behavior of the relativistic diffusion in de Sitter and anti de Sitter space-times, seen in Iwasawa coordinates. 

\begin{prop}\label{prop.sumup}
The relativistic diffusion $(\xi_s, \dot{\xi})_{s\geq 0}$ in $\dS$ or $\AdS$, written in Iwasawa coordinates, corresponds bijectively to the diffusion $ (n_s,a_s, \theta_s)_{s\geq 0} \in N \times A \times K/\mathrm{SO}(d)$. 
\begin{itemize}
\item The process $(\theta_s)_{s\geq0}$ is an ergodic subdiffusion, taking values in $\mathbb{S}^d$  in the case of $\dS$,  and in $\mathbb{S}^{1}$ in the case of $\AdS$. 
\item The process $(\theta_s, a_s)_{s\geq0}$ is also a subdiffusion, where $(a_s)_{s\geq0}$ takes values in $\mathcal{A}$ identified with $\R$ for $\dS$ or $\R^2$ for $\AdS$. As $s$ goes to infinity, $\frac{1}{s}\log a_s$ converges almost surely to a deterministic point $\eta_{\infty}$ of the Weyl chamber $\mathcal{A}^{+}$.
\item The process $(n_s)_{s\geq0}$, with values in $N$ identified to $\R^{d}$ for $\dS$ or $\R^{2d-2}$ for $\AdS$, converges almost surely as $s$ goes to infinity to an asymptotic random variable $n_\infty \in N$. 
\end{itemize}
\end{prop}

\subsubsection{Geometric identification of the boundary} \label{sec.geomeaning}
Let us now describe with more geometrical details the almost sure asymptotic behavior of the relativistic diffusion $(\dot{\xi}_{s}, \xi_s)_{s\geq0}$ in $\dS$ and $\AdS$ space-times.

\paragraph{Asymptotics in de Sitter space-time.}The de Sitter space-time is a double cover of an open set of the projective space $\PR$ via the map 
\[
\begin{matrix}
p:&\dS & \longrightarrow & \PR \\
&\xi & \longmapsto & \mathrm{vect}(\xi).
\end{matrix}
\]
In the affine chart given by the hyperplan $\{\xi \in \R^{1,d+1}, \xi_0=1 \}$ the image of the isotropy cone of $Q_{1,d+1}$ is a sphere and the image of $\dS$ correspond to the exterior of that sphere which we denote by $\partial \dS$ and corresponds to the causal and conformal boundary of $\dS$. The Lie algebra $o(1,d+1)$ of $\mathrm{PSO}(1,d+1)$ is by definition 
\[
o(1,d+1) = \left \{   \left (\begin{matrix} 0 & u^* \\ u & A   \end{matrix} \right ) , \  u \in \R^{d+1}, \ A \in \R^{(d+1)\times(d+1)}, \ A^{*}=-A  \right \}
\]
and the Iwasawa decomposition $o(1,d+1)= \mathcal{N}\oplus \mathcal{A} \oplus \mathcal{K}$ we choose is precisely
\begin{align*}
\mathcal{A}= &\left \{ \left (\begin{matrix} 0 & \beta & 0 \\ \beta & 0 & 0 \\ 0 & 0 & 0  \end{matrix} \right), \beta \in \R  \right \}, \ \ 
 \mathcal{N}= \left \{ \left (\begin{matrix} 0 & 0 & h^{*} \\ 0  & 0 & -h^{*} \\ h & h &  0 \end{matrix} \right), h \in  \R^d \right \}, \\
& \  \mathrm{and} \quad  \mathcal{K}= \left \{   \left( \begin{matrix} 0 & 0  \\ 0 & A  \end{matrix} \right), \ A \in \R^{(d+1)\times (d+1)} ,\ A^{*}=-A  \right \}.  
\end{align*}
Thus the Iwasawa decomposition of $g_s$ writes explicitely
\[
g_s = n_s a_s k_s = \exp \left ( \begin{matrix} 0 & 0 & h_s^{*} \\ 0  & 0 & -h_s^{*} \\ h_s & h_s &  0  \end{matrix} \right ) \left ( \begin{matrix}  \cosh(\beta_s) & \sinh(\beta_s) & 0 \\ \sinh(\beta_s) & \cosh(\beta_s) & 0 \\ 0 & 0 & \mathrm{I}_d \end{matrix} \right ) \left ( \begin{matrix}1 &0 \\ 0& R_s \end{matrix}\right)
\]
where $R_s \in \mathrm{SO}(d+1)$. 
Then, we get that $(\dot{\xi}_s, \xi_s):= \pi(g_s)=(g_s(e_0), g_s(e_1) )=(n_s a_s (e_0), n_s a_s \theta_s) $, where by definition $\theta_s:= k_s(e_1)$ corresponds to the first column of $R_s$ and belongs to the set $\{\xi \in \R^{1,d+1}, \xi^0 =0 \ \mathrm{and} \ Q_{1,d+1}(\xi)=1 \} \simeq \mathbb{S}^{d} $. From the above Proposition \ref{prop.sumup}, we know that almost surely $\beta_s$ goes to infinity with $s$ and since $\theta_s$ is bounded, we get
\begin{align*}
Q_{1,d+1}(n_s(e_0+e_1), \xi_s) &= Q_{1, d+1}(n_s(e_0+e_1), n_s a_s \theta_s) \\
&=Q_{1,d+1}(a_{s}^{-1}(e_0+e_1), \theta_s )  \\
&= e^{-\beta_s} Q_{1,d+1}(e_0+e_1, \theta_s) \underset{t \to +\infty}{\longrightarrow} 0.
\end{align*}
From a geometrical point of view, this means that in the projective space, all adherent point of $p(\xi_s)$ belongs to the set of points in $\PR$ which are $Q_{1,d+1}$-orthogonal to the limit point $p(n_\infty(e_0+e_1))$. This set is the tangent space of the sphere $\partial \dS$ at the point $p(n_{\infty}(e_0+ e_1) )$. But by definition $p(\xi_s)$ is a causal curve and converges to a point of the causal boundary which is also $\partial \dS$. This implies necessary that $p(\xi_s)$ converges to $p(n_{\infty}(e_0+e_1))$. Note that one can also recover this result writing $\theta_s= \sum_{i=1}^{d+1} \theta_s^{i} e_i$ so that
\[
\xi_s = \theta^{1}_s n_s \left ( \sinh(\beta_s) e_0 + \cosh(\beta_s) e_1 \right ) + \sum_{i=2}^{d+1} \theta_s^{i} n_t(e_i).
\]
In \cite{camille}, it is shown that $(\theta^{1}_s)_{s\geq0}$ is a diffusion process taking values in $[-1,1]$ which stays eventually in $]0,1[$, implying that, for $s$ large enough, we have $\theta_s^{1}\neq0$. Since $\beta_s$ goes to infinity with $s$, one gets
\[
\left (\theta^{1}_s e^{\beta_s}\right )^{-1}\xi_s \underset{s \to +\infty}{\longrightarrow} n_{\infty}(e_0+e_1), \;\; \text{and thus} \;\;
p(\xi_s)= p\left (\left (\theta^{1}_s e^{\beta_s}\right )^{-1} \xi_s\right ) \underset{s\to +\infty}{\longrightarrow} p( n_{\infty}(e_0+e_1) ).
\]
Let now describe the asympotic behavior of $(\dot{\xi}_{s})_{s\geq0}$. We have $\dot{\xi}_s= n_s (\cosh(\beta_s)e_0+\sinh(\beta_s)e_1)$ and so $e^{-\beta_s}\dot{\xi}_s$ converges to $n_{\infty}(e_0 +e_1)$. Thus 
\[
p(\dot{\xi}_s)= p(e^{-\beta_s}\dot{\xi}_s)  \underset{s \to +\infty}{\longrightarrow}  p(n_{\infty}(e_0 +e_1)).
\]
Geometrically, the almost sure asymptotic behavior of $(\xi_s, \dot{\xi}_{s})_{s\geq0}$ in the projective space in thus the one described in Figure \ref{fig.dS3} below.
\begin{figure}[ht]
\begin{center}
\includegraphics[scale=0.5]{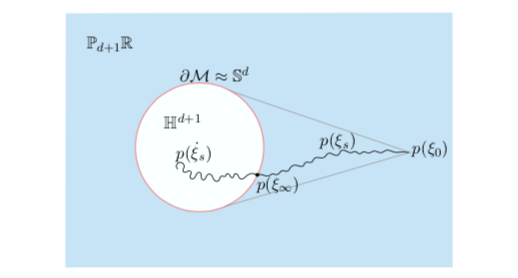}
\end{center}
\caption{Asymptotics of the relativistic diffusion under the hypotheses of Theorem \ref{theo.infhorEucli}.}\label{fig.dS3}
\end{figure}

\begin{rema}\label{rem.dS}
For every $\Theta \in \partial \dS$ which is different to $p(e_0-e_1)$ there is a unique $n\in N$ such that $\theta = p(n(e_0+e_1))$. Thus, Theorem \ref{theo.poisson.dS} can be rephrased saying that the invariant $\sigma$ field of $(\dot{\xi}_s, \xi_s)_{s\geq0}$ is generated by the asymptotic point $\Theta_{\infty}:=p(n_{\infty}(e_0+e_1))$ belonging to the causal boundary of $\dS$.
\end{rema}

\paragraph{Asymptotics in Anti de Sitter space-time.}  Recall that by definition, the Anti de Sitter space-time is the algebraic submanifold  $\AdS= \{ \xi \in \R^{2,d}, \ Q_{2,d}(\xi)=-1 \}$ and let us denote by $\Ein:=\{ \xi \in \R^{2,d}, \ Q_{2,d}(\xi)=0 \}$ which inherits from $Q_{2,d}$ a flat conformal Lorentzian structure and is called the Einstein Lorentz manifold. Denote by $p:\R^{2,d}\setminus \{0\} \to \PR$ the projection onto the projective space. Then $p(\AdS)$ is an open set whose boundary is $p(\Ein)$, and seen in the affine chart $\{\xi \in \R^{2,d}, \xi^{0}=1 \}$, $p(\AdS)$ is the interior of a one-sheeted hyperboloid corresponding to $p(\Ein)$. 
The Lie algebra $o(2,d)$ of $\mathrm{PSO}(2,d)$ is the set of matrices
\[
o(2,d)= \left \{ \left ( \begin{matrix}  0 & a & b^* \\ -a & 0 & c^* \\ b & c & A \end{matrix}  \right ), \ a \in \R, \ b,c \in \R^d , \ A\in \R^{d\times d}, \ A^{*}=-A \right \},
\]
and the Iwasawa decomposition $o(2,d)=\mathcal{N}\oplus \mathcal{A} \oplus \mathcal{K}$ we choose is given by 
\[
\mathcal{K}= \left \{ \left ( \begin{matrix} 0 & -\theta & 0 \\ \theta & 0 & 0 \\ 0 &0 & A \end{matrix} \right ), \  A\in \R^{d\times d}, \ A^{*}=-A, \ \theta \in \R  \right \}, \ 
\] 
\[
\mathcal{A}= \left \{ \left (\begin{array}{cc|ccc} 0 & 0 & \lambda & 0 & 0 \\ 0 & 0 & 0 & \mu & 0 \\ \hline \lambda & 0 & 0 & 0 & 0 \\ 0 & \mu & 0 & 0 & 0 \\ 0 & 0 & 0 & 0 & 0 \end{array} \right) \lambda, \mu \in \R \right  \},
\]
and the nilpotent Lie algebra
\[
\mathcal{N}= \left  \{ \left (\begin{array}{cc|ccc} 0 & -(x+y) & 0 & x-y & h^* \\ x+y & 0 & x+y & 0 & \tilde{h}^* \\ \hline 0 & x+y & 0 & -(x-y) & -h^* \\ x-y & 0 & x-y & 0 & -\tilde{h}^{*} \\ h & \tilde{h} & h & \tilde{h} & 0 \end{array} \right), \  x,y \in \R, \  h,\tilde{h} \in \R^{d-2} \right \}.
\]
The Weyl chamber is the set of elements of $\mathcal{A}$ such that $\lambda >0$, $\mu >0$ and $\lambda-\mu >0$.
Remark that $\dim \mathcal{A} = 2$, $\dim \mathcal{K}=1+ \frac{d(d-1)}{2}$ and $\dim \mathcal{N}=2d-2$. Denote by $N$, $A$ and $K$ the Lie groups associated so that $\mathrm{PSO}(2,d)= NAK$. Since $T^1_{+}\AdS$ is identified with $\mathrm{PSO}(2,d)/\mathrm{SO}(d)$ one has the Iwasawa diffeomorphism
\[
\begin{matrix}
N \times A \times \mathbb{S}^1 & \longrightarrow & T^1 \AdS \\ 
(n,a,\theta) & \longmapsto & \left (na(\cos(\theta)e_0 + \sin(\theta)e_1),  na ( -\sin(\theta)e_0 + \cos(\theta)e_1) \right) . 
\end{matrix}
\]
Thus in Iwasawa coordinates the relativistic diffusion writes 
\begin{align*}
\xi_s= n_s a_s (\cos(\theta_s)e_0 + \sin(\theta_s)e_1), \quad \dot{\xi}_s= n_s a_s ( -\sin(\theta_s)e_0 + \cos(\theta_s)e_1).
\end{align*}
Define the processes $(\lambda_s)_{s\geq0}$ and $(\mu_s)_{s\geq0}$ so that 
\[
a_s= \exp\left (\lambda_s (e_2e_0^*+e_0e_2^*) + \mu_s (e_3e_0^*+ e_0e_3^*)\right), \;\; \forall \; s\geq 0.
\] 
By Proposition \ref{prop.sumup}, since $\log(a_s)/s$ converges to a deterministic element of the interior of the Weyl chamber, one deduces that $\lambda_s/s\to \lambda_{\infty}>0$, $\mu_s/s\to \mu_{\infty}>0$ and $\lambda_{\infty} > \mu_{\infty}$.
Let us now define
\[
u_s:= e^{-\lambda_s}n_s a_s (e_0), \quad v_s:= e^{-\mu_s}n_s a_s (e_1),
\]
so that we have 
\[
\xi_s= \cos(\theta_s)e^{\lambda_s} u_t + \sin(\theta_s)e^{\mu_s} v_s, \quad \text{and} \quad \dot{\xi}_s= -\sin(\theta_s)e^{\lambda_s}u_s + \cos(\theta_s)e^{\mu_s}v_s. 
\]
By Proposition \ref{prop.sumup} again, we have then almost surely
\[
u_s \underset{s \to +\infty}{\longrightarrow} n_\infty (e_0+e_2), \quad 
v_s \underset{s \to +\infty}{\longrightarrow} n_{\infty}(e_1+e_3). 
\]
Thus, in the manifold $\R^{2,d}$, the plan $\mathrm{vect}(\dot{\xi}_s,\xi_s)=\mathrm{vect}(u_s,v_s)$ converges to the asymptotic plan $\mathrm{vect}(n_\infty (e_0+e_2), n_{\infty}(e_1+e_3))$ belonging to the isotropy cone of $Q_{2,d}$.  Choosing $(u_s,v_s)$ as a fixed basis of $\mathrm{vect}(\dot{\xi}_s,\xi_s)$, the two points $\xi_s$ and $\dot{\xi}_{s}$ describe spirals and go to infinity faster in the direction given by $u_s$ (since $\lambda_s-\mu_s$ goes to infinity with $s$).
Looking in the projective space $\PR$ the line $p(\mathrm{vect}(\xi_s, \dot{\xi_s}))=p(\mathrm{vect}(u_s,v_s))$ then converges as $s$ goes to infinity to the asymptotic line $\ell_\infty:=p(\mathrm{vect}(n_\infty(e_0+e_2), n_\infty(e_1+e_3)))$ which belongs to $\Ein$ and is a \emph{light circle} (more precisely it is a light-like geodesic belonging to a light cone in $\Ein$). 
\par
\bigskip
One can easily check that $n_\infty$ is determinated by $n_\infty (e_0+e_2)$ and $n_{\infty}(e_1+e_3)$. In fact, given a point $\eta \in \Ein$ and a light-like geodesic $\ell \subset \Ein$ passing by $\eta$ such that $\ell$ does not belong to the light cone of $\Ein$ with apex $p(e_0-e_2)$, there is a unique $n\in N$ such that $p(\eta)=p(n(e_0+e_2))$ and $\ell= p(\mathrm{vect}(n(e_0+e_2), n(e_1+e_3)))$. This means that $N$ is diffeomorphic to the set of pointed light circle in $p(\Ein)$, which does not belong to the light cone with apex $p(e_0-e_1)$. Thus the set of exit points of the relativistic diffusion in $\AdS$, encoded by $n_{\infty}$, is the pointed light circle $(p_\infty, \ell_\infty)$, given by $p_{\infty}:= p(n_{\infty}(e_0+e_2))$ and $\ell_{\infty}:= p(\mathrm{vect}(n_{\infty}(e_0+e_2), n_{\infty}(e_1+e_3) ))$. Therefore, even if $p(\dot{\xi}_s)$ and $p(\xi_s)$ are asymptotic to the random light circle $\ell_{\infty}$, the relativistic diffusion carries an extra invariant information given by the point $p_{\infty} \in \ell_{\infty}$.

\subsubsection{Expliciting the Poisson boundary} \label{sec.poisson.model}
With the help of the d\'evissage method, more precisely with the help of Theorem \ref{theo.trivial} of Section \ref{sec.devissage}, we are now in position to make explicit the Poisson boundary of the relativistic Brownian motion in de Sitter and Anti de Sitter space-times. The starting point is Proposition \ref{prop.sumup} describing the long time asymptotic behavior of the diffusion in Iwasawa coordinates. Namely,  we have seen in Section \ref{sec.iwa} that $X_s:=(a_s,\theta_s)_{s\geq 0}$ is a subdiffusion of $(X_s,Y_s)_{s\geq 0}:=(a_s,\theta_s, n_s)_{s\geq0}$ ; moreover, by Equation \eqref{eq.iwa.n}, the dynamics of the process $Y_s:=n_s$ depends only of $(a_s,\theta_s)_{s\geq 0}$ so that the equivariance condition is satisfied (the group $N$ acting on itself by translation) and it converges almost surely as $s$ goes to infinity to $Y_{\infty}=n_{\infty}$. Last, the generator of the relativistic diffusion being hypoelliptic, harmonic functions are smooth. Therefore, all the d\'evissage conditions are fulfilled and 
by Theorem \ref{theo.trivial}, we get that $\mathrm{Inv}((a_s,\theta_s, n_s)_{s \geq 0})$ and $\sigma(n_\infty) \vee \mathrm{Inv}((a_s, \theta_s)_{s \geq 0})$ coincide up to negligeable sets.
\par \bigskip
It remains to check that the Poisson boundary of $(a_s,\theta_s)_{s\geq0}$ is trivial. We will actually prove that the Poisson boundary of the process $(a_s,k_s)_{s\geq0}$ is trivial. Recall that $\tilde{\Op}$ is the generator of the lifted diffusion $(g_s=n_s a_s k_s)_{s \geq 0}$ introduced in Section \ref{sec.lift} so that every harmonic function for the diffusion $(a_s,k_s)_{s\geq 0}$ is a $\tilde{\Op}$-harmonic function which is constant in the $N$-variable. So, in order to verify that the Poisson boundary of $(a_s,k_s)_{s\geq0}$ is trivial it remains to show that every bounded $\tilde{\Op}$-harmonic function which is constant in the $N$-coordinates (or equivalently is $N$-left invariant) is constant. For that we proceed as in \cite{harry1}, \cite{raugi77}, \cite{bailleul_raugi} or \cite{camille} and show that a $N$-left invariant, $\tilde{\Op}$-harmonic function is necessarily $A$-left invariant. So it is harmonic for the generator of the ergodic subdiffusion $(k_s)_{s\geq0}$ and thus is constant. 

\begin{prop}\label{trivial}
The invariant $\sigma$-field of the subdiffusion $(a_s,k_s)_{s\geq0}$ is trivial a.s. 
\end{prop}
\begin{proof}
Let us recall the following classical fact about bounded harmonic functions of left invariant hypoelliptic diffusions on Lie groups (cf Prop 3.4 of \cite{bailleul_raugi}). Denote by $S$ the support of $\mu_1$ the law of $g_1$ under $\Prob_{\mathrm{Id}}$. Note that the process $(g_n)_{n\in \N}$ is a left random walk on $G$ generated by $\mu_1$. 
\begin{lemm} \label{classical}
Every bounded harmonic function $f$ is right uniformly continuous and for every $s \in S$, $f(g_n s)$ converges $\Prob_{g}$-a.s and
\[
\lim_{n \to +\infty} f(g_n s) = \lim_{n \to +\infty} f(g_n).
\]
\end{lemm}
Take now $f$ a bounded $\tilde{\Op}$-harmonic function which is $N$-left invariant (or equivalently a bounded harmonic function for the subdiffusion $(a_s, k_s)_{s \geq 0}$). We want to show that $f$ is constant and for that it is sufficient to  prove that it is $A$-left invariant ; indeed if it is the case, $f$ depends only on the variable in $K$ and so is harmonic for the ergodic diffusion $(k_s)_{s\geq 0}$ and thus is constant. Since $A$ is abelian (and isomorphic to either $\R$ or $\R^2$ here) it is sufficient to verify that 
$
f(ak)=f(k)
$
holds for every $k\in K$ and every $a$ in some neighborhood $\mathcal{V}_{\mathrm{Id}}$ of $\mathrm{Id}$ in $A$. We denote by $S$ the support of the law of $g_1$ under $\Prob_{\mathrm{Id}}$. By hypoellipticity, the set $S$ contains some open set and thus $SS^{-1}$ contains a neighborhood $\mathcal{U}$ of $\mathrm{Id}$ in $G$. Then we take  $\mathcal{V}_{\mathrm{Id}}$ small enough such that 
\[
\{ k a k^{-1}, \quad k \in K, \ a \in \mathcal{V}_{\mathrm{Id}} \} \subset \mathcal{U}.
\]
Now fix $a\in \mathcal{V}_{\mathrm{Id}}$ and choose $k_{\infty}$ belonging to the support of the invariant law of $(k_{s})_{s\geq0}$ and set $u_{\infty}:= k_{\infty}^{-1} a k_{\infty}$ which belongs to $\mathcal{U}$ and thus to $SS^{-1}$. So there exists $s, s'\in S$ such that $u_{\infty}= s's^{-1}$. Then, denoting by $\E_{(a,k)}$ the expectation over all trajectories $(a_s,k_s)_{s\geq 0}$ starting at $(a,k)\in A\times K$, one gets
\begin{align*}
f(ak)&= \E_{(a,k)} \left[ \lim_{n \to +\infty } f(a_n k_n) \right]= \E_{(a,k)}\left[\lim_{n \to +\infty } f(a_n k_n s)\right] \quad (\mathrm{using} \ \mathrm{Lemma} \ \ref{classical}) \\
&=\E_{(\mathrm{Id},k)} \left[ \lim_{n \to +\infty } f(a a_n k_n s) \right] = \E_{(\mathrm{Id},k)} \left[ \lim_{ n \to +\infty } f( a_n k_n u_n s) \right]
\end{align*}
where we set $u_n:= k_n^{-1} ak_{n}$. By ergodicity of $(k_s)_{s\geq0}$, one deduces that for almost all trajectories it exists a (random) subsequence $(n_m)_{m}$ such that $u_{n_m}$ converges to $u_{\infty}$. Now, since $f$ is right uniformly continuous, and since $u_{\infty}= s's^{-1}$, it comes that for $\Prob_{(\mathrm{Id},k)}$-almost every trajectories
\[
\lim_{ n \to +\infty } f( a_n k_n u_n s) = \lim_{ m \to +\infty } f( a_{n_m} k_{n_m} u_{\infty} s) = \lim_{ m \to +\infty } f( a_{n_m} k_{n_m} s').
\]
But since $s' \in S$ one finally obtains, by Lemma \ref{classical} again, that 
\[
\Prob_{(\mathrm{Id},k)}-a.s, \quad \lim_{ n \to +\infty } f( a_n k_n u_n s) = \lim_{n \to +\infty} f(a_n k_n).
\]
Thus taking expectation, one gets
\begin{align*}
f(ak) = \E_{(\mathrm{Id}, k)} \left[\lim_{n \to +\infty} f(a_n k_n)\right] = f(k),
\end{align*}
and the proof is completed. 
\end{proof}
\begin{coro}
The invariant $\sigma$-field of the relativistic diffusion $(\dot{\xi}_s, \xi_s)_{s \geq0}$ with values in de Sitter or Anti de Sitter space-times is generated by the asymptotic random variable $n_{\infty} \in N$.
\end{coro}
\begin{proof}
Recall that $(\dot{\xi}_{s}, \xi_s)_{s\geq 0}$ is in bijection with $(n_s, a_s, \theta_s)_{s\geq0}$ so that 
\[
\mathrm{Inv}((\dot{\xi}_s, \xi_s))_{s\geq0})= \mathrm{Inv}((n_s, a_s,\theta_s)_{s\geq0}).
\]
By the d\'evissage method and Proposition \ref{trivial}, since $ (a_s,\theta_s)_{s\geq0}$ is a subdiffusion of $(a_s,k_s)_{s\geq0}$, we deduce that up to negligeable sets
\begin{align*}
\mathrm{Inv}((\dot{\xi}_s, \xi_s))_{s\geq0}) & = \sigma(n_\infty) \vee \mathrm{Inv}((a_s,\theta_s)_{s\geq0})= \sigma(n_\infty) .
\end{align*}
\end{proof}

\begin{rema}
Let us summarize here the results of Section \ref{sec.model}. In model space-times, the set of exit points of the relativistic diffusion is thus in bijection with the group $N$ appearing in the Iwasawa decomposition of the orthonormal frame bundle of the base manifold. In the case of de Sitter space-time, by Remark \ref{rem.dS}, the group $N$ identifies with the causal / conformal boundary $\partial \dS \approx \mathbb S^d$. In the case of Anti de Sitter space-time, as detailed at the end of Section \ref{sec.geomeaning} above, the group $N$ this time identifies with the set of pointed light circles, carrying more information than the single conformal boundary.
\end{rema}
\subsection{Poisson boundary in Robertson--Walker space-times} \label{sec.RWglobal}

Having given a full description of the long time asymptotics of the relativistic diffusion in the three Lorentz model manifolds, we now turn to the case of curved manifolds and more precisely the case of Robertson--Walker space-times. In this geomeric context, the almost sure long time asymptotic behavior of the relativistic diffusion was described in details in \cite{mathese,angst2016} and the main point here is thus to explicit the Poisson boundary.  Let us first recall here some basic facts that are established in the latter references. In a warped product $\mathcal M = (0,+\infty) \times_{\alpha} M$, the relativistic diffusion  $(\xi_s, \dot{\xi}_s)$ with values in $T^1 \mathcal M$ can naturally be written $(\xi_s, \dot{\xi}_s)=(t_s,\dot{t}_s, x_s, \dot{x}_s)$ with $(t_s, \dot{t}_s) \in T (0, +\infty)$ and $(x_s, \dot{x}_s) \in TM$, and the fact that the process $\dot{\xi}_s$ has unit pseudo-norm translates into the relation 
\begin{equation}\label{pseudonorm}
\dot{t}_s^2 - 1 = \alpha^2(t_s)|\dot{x}_s|^2.
\end{equation}
The ``temporal'' process $(t_s, \dot{t}_s)$ then appears as a two-dimensional subdiffusion, solution of the following system of Equations 
\[
d t_s = \dot{t}_s ds, \qquad d\dot{t}_s =-H(t_s) (\dot{t}_s^2-1)ds + \frac{d \sigma^2}{2} \dot{t}_s ds  + \sigma d M^{\dot{t}}_s,
\]
where $M^{\dot{t}}_s$ is a local martingale with bracket $d \langle M^{\dot{t}},M^{\dot{t}}\rangle_s= (\dot{t}_s^2-1)ds$.
In other words, it admits the following hypoelliptic infinitesimal generator 
\[
L_{(t, \dot{t})} := \dot{t} \partial_t + \left(-H(t) (\dot{t}^2-1)ds + \frac{d \sigma^2}{2} \dot{t} \right) \partial_{\dot{t}} + \frac{ \sigma^2}{2}(\dot{t}^2-1) \partial_{\dot{t}}^2.
\]
Lemma 4.1 of \cite{angst2016} then ensures that almost surely $\dot{t}_s^2 - 1>0$ for all $s>0$, i.e. $|\dot{x}_s|>0$, allowing to consider the process $(x_s, \dot{x}_s/|\dot{x}_s|) \in T^1M$. The trajectories of the process being future oriented, we have in fact $\dot{t}_s>1$ for all $s>0$ which ensures that the first component $t_s$ is transient. The recurrence/transience of the derivative $\dot{t}_s$ is well understood and is related to the rate of decrease of the Hubble function $H$, namely $\dot{t}_s$ is almost surely transient if and only if $H^d$ is integrable at infinity, see Proposition 3.3 of \cite{angst2016}. 
No matter the recurrence or transience of the component $(\dot{t}_s)_{s \geq 0}$, the following Liouville Theorem for the temporal subdiffusion was established in \cite{angst1}.

\begin{prop}[Proposition 4.8 of \cite{angst1}]\label{pro.temptrivial}
The invariant sigma field associated to the diffusion $(t_s, \dot{t}_s)$ is trivial, equivalently all bounded $L_{(t, \dot{t})}-$harmonic functions are constant. 
\end{prop}

\begin{rema}\label{rem.indepcoupling}
The proof of the above Proposition \ref{pro.temptrivial} given in the reference \cite{angst1} actually shows that there exists an automatic coupling between two independant copies of the temporal subdiffusions. This fact will be used in the proofs in the next Propositions \ref{pro.etrivial} and \ref{prop.predeviss}.
\end{rema}

As in the case of geodesics, the long time asymptotic behavior of the ``spatial'' components $(x_s, \dot{x}_s/|\dot{x}_s|) $ of the relativistic diffusion is then intimately related to the finiteness of the two integrals
\begin{equation}\label{eq.defCD}
C_s:=\int_0^{s} \dot{C}_u du, \quad\hbox{and} \quad  D_s:=\int_0^{s} \dot{D}_u du,
\end{equation}
where 
\begin{equation}\label{eq.defCD2}
\dot{C}_s := \frac{1}{\dot{t}_s^2-1} \qquad \dot{D}_s :=\frac{\sqrt{\dot{t}_s^2-1}}{\alpha(t_s)}.
\end{equation}
The finiteness of these two random additive functionals is related to the geometry of the manifold in the following way. 

\begin{prop}[Corollary 4.2 of \cite{angst2016}]\label{myprop}
If the torsion function $\alpha$ satisfies the hypotheses of Section \ref{sec.geoRW}, then we have the equivalences
\[
C_{\infty}<+\infty \;\; \text{a.s.} \Longleftrightarrow \begin{cases} H^3 \in \mathbb{L}^1, &\mathrm{if} \ d \geq 4 \\ H^3 \in \mathbb{L}^{1-} \ &\mathrm{if} \ d=3\end{cases} , \quad D_{\infty}<+\infty \;\; \text{a.s.}  \Longleftrightarrow I(\alpha)<+\infty.
\]
\end{prop}

Before getting into more details for each cases let us before recall that the pseudo-norm relation \eqref{pseudonorm} insures that the relativistic diffusion $(\xi_s, \dot{\xi}_s)_{s\geq 0}$ is in bijection with the diffusion $(t_s, \dot{t}_s, x_s, \dot{x}_s/\vert \dot{x}_s\vert)_{s\geq0}$ taking values in $T(0,+\infty) \times T^{1}M$. When $M= \mathbb R^d$ we use the canonical coordinates $(x^1, \dots , x^d)$ and when $M=\mathbb{S}^d$ or $M=\mathbb{H}^d$ we see it as Riemannian submanifolds of $\mathbb R^d$ and $\mathbb R^{1,d}$ respectively and we use extrinsic coordinates $(x^0, x^1, \dots, x^d)$. The dynamics of $(x_s, \dot{x}_s/{\vert \dot{x}_s\vert}) \in T^1 M $ can be obtained from \eqref{eqn.flj} and is given by 
\begin{equation}\label{eqn.xdot}
\left \{
\begin{array}{rl}
\displaystyle{ d x_s} = &\displaystyle{\dot{D}_s \frac{\dot{x}_s}{\vert \dot{x}_s \vert} ds,} \\
\\
  \displaystyle{d \frac{\dot{x}^{\mu}_s}{\vert \dot{x}_s \vert } }= & \displaystyle{- \kappa \dot{D}_s x^{\mu}_s ds - \sigma^2 \frac{d-1}{2} \dot{C}_s \times \frac{\dot{x}^{\mu}_s}{\vert \dot{x}_s \vert } ds + \sigma \sqrt{\dot{C}_s}dM_{s}^{\dot{x}^{\mu}/\vert \dot{x} \vert} \quad \mu=0,\dots, d,}
 \\
\end{array}
\right.
\end{equation}
where we exclude the coordinate $\mu =0$ when $M=\mathbb R^d$. The parameter $\kappa$ denotes the constant curvature of $M$ and thus takes values $0, 1$ and $-1$ if $M=\mathbb R^d$, $\mathbb S^d$ or $\mathbb H^d$ respectively. The martingale terms $M_{s}^{\dot{x}^{\mu} /\vert x \vert}$ appearing in Equation \eqref{eqn.xdot} are such that 
\begin{equation*}
\left \{
\begin{array}{rll}
\displaystyle{d \langle M^{\dot{x}^{\mu}/\vert \dot{x} \vert}, M^{\dot{t}}_s \rangle} &=0, & \mu=0,\dots, d, \\
\\
\displaystyle{d \langle M^{\dot{x}^{\mu}/\vert \dot{x} \vert}, M^{\dot{x}^{\nu}/\vert \dot{x} \vert} \rangle_s} &= \displaystyle{\big (\delta_{\mu \nu} + (\kappa -1)\delta_{\mu 0}\delta_{\nu 0} - \kappa x^{\mu}_s x^{\nu}_s -\frac{\dot{x}^{\mu}_s}{\vert \dot{x}_s \vert }\frac{\dot{x}^{\nu}_s}{\vert \dot{x}_s \vert } \big )ds} & \displaystyle{\mu,\nu=0,\dots,d.}
\end{array}
\right.
\end{equation*}
Here again, in the case where $M=\mathbb R^d$ or equivalently $\kappa =0$, the indices $\mu, \nu=0$ are excluded. 
The reader familiar with diffusion processes on manifolds will recognise in the dynamics of $\dot{x}_s/\vert \dot{x}_s \vert$ the one of a spherical Brownian motion on the unit spheres tangent to $x_s$ and changed in time by the clock $C_s$ defined in Equations \eqref{eq.defCD} and \eqref{eq.defCD2}.

\subsubsection{Robertson--Walker space-times with finite horizon} \label{sec.RWfinite}
This section is devoted to the proof of Theorem \ref{theo.finitehori} which corresponds to the case where the space-time $\mathcal M = (0,+\infty) \times_{\alpha} M$ has finite horizon, i.e. the case where $I(\alpha)<+\infty$. Then, the asymptotic behavior of the relativistic diffusion obbeys a dichotomy depending on the integrability of the Hubble function, namely if the decreasing function $H^3$ is integrable at infinity or not.

\paragraph{The case where $H^3 \in \mathbb L^1$} \textcolor{white}{blanc}\par
In the case where the space-time $\mathcal M = (0,+\infty) \times_{\alpha} M$ has finite horizon and the Hubble function satisfies the condition $H^3 \in \mathbb L^1$, then Theorem 3.3 of \cite{angst2016} makes explicit the almost sure asymptotic behavior of the relativistic diffusion $(\xi_s, \dot{\xi}_s)=(t_s,\dot{t}_s, x_s, \dot{x}_s)$, precisely
\begin{itemize}
\item the temporal process $(t_s, \dot{t}_s)$ is transient almost surely,
\item  the spatial process $(x_s, \dot{x}_s/|\dot{x}_s|)$ with values in $T^1 M$  converges almost surely to a random point in $ T^1 M$. 
\end{itemize}
Note that the unit tangent bundles of the Riemannian fibers $M=\mathbb R^d$, $\mathbb H^d$ or $\mathbb S^d$ can all be expressed as homogeneous spaces of the form $G/K$, namely we have
\[
T^1\mathbb R^d = \mathbb R^d \rtimes \mathrm{SO}(d)/\mathrm{SO}(d-1), \quad T^1\mathbb H^d = \mathrm{PSO}(1,d)/\mathrm{SO}(d-1), \quad T^1 \mathbb S^d = \mathrm{SO}(d+1)/\mathrm{SO}(d-1).
\]
Moreover, the dynamics of the relativistic diffusion is easily seen to be equivariant under the action of $K=\mathrm{SO}(d-1)$, so that we are precisely in position to apply the d\'evissage scheme with
\[
X_s:=(t_s, \dot{t}_s) \in X:=(0,+\infty) \times [1, +\infty), \quad Y_s = (x_s, \dot{x}_s/|\dot{x}_s|) \in G/K.
\]
Namely, applying Theorem \ref{theo.homo} recalled in Section \ref{sec.devissage}, we get the almost sure identification
\[
\text{Inv}(t_s, \dot{t}_s,\dot{x}_s/|\dot{x}_s|,x_s) = \text{Inv}(t_s, \dot{t}_s) \vee \sigma(Y_{\infty}).
\]
From the above Proposition \ref{pro.temptrivial}, we know that the invariant sigma field $\text{Inv}(t_s, \dot{t}_s)$ is trivial, so we can conclude that $
\text{Inv}(t_s, \dot{t}_s,\dot{x}_s/|\dot{x}_s|,x_s) = \sigma(Y_{\infty}),$ up to negligeable sets, which is the first claim of Theorem \ref{theo.finitehori}. 

We now give the proof of the second claim of Theorem \ref{theo.finitehori}. We give two separate proofs, the first one   corresponding to the case of a warped product with a Euclidean fiber, the second one allowing to deal with non zero curvature fibers. 

\paragraph{The case where $H^3 \notin \mathbb L^1$ and the fiber is Euclidean.} We consider here the case of a warped product $\mathcal M = (0,+\infty) \times_{\alpha} \mathbb R^d$ under the hypothesis that $H^3 \notin \mathbb L^1$. In virtue of Proposition 4.6 and the proof of Proposition 4.7 in \cite{angst2016} and Lemma 4.1 of \cite{angst1}, the asymptotic behavior of the relativistic diffusion $(\xi_s, \dot{\xi}_s)=(t_s,\dot{t}_s, x_s, \dot{x}_s)$ is then the following: 
\begin{itemize}
\item the temporal subdiffusion $(t_s,\dot{t}_s)$ is transient.
\item the process $(t_s,\dot{t}_s, \dot{x}_s/|\dot{x}_s|)$ is also a subdiffusion of the whole process.
\item there exists a Brownian motion $\widetilde{\Theta}_s$ with value in $\mathbb S^{d-1}$, which is independant of the temporal subdiffusion and such that normalized angular derivative writes $\dot{x}_s/|\dot{x}_s| = \widetilde{\Theta}(C_s)$. 
\item the projection $x_s \in \mathbb R^d$ converges almost surely to a random limit point  $x_{\infty} \in \mathbb R^d$.
\end{itemize}
Let us now show how these results imply that the Poisson boudary of the full relativistic diffusion coincides with $\sigma(x_{\infty})$. Let us first prove the following Liouville theorem 
\begin{prop}\label{pro.etrivial}
The invariant sigma field associated to the diffusion $(t_s, \dot{t}_s,\dot{x}_s/|\dot{x}_s|)$ is trivial. 
\end{prop}
\begin{proof}
The proof is  similar to the one of Proposition 4.9 in \cite{angst1}, for the sake of self-containess, we reproduce it here. 
Following \cite{cranston}, the triviality of the Poisson boundary is equivalent to the existence of a shift coupling between two sample paths of the process. To simplify the expressions, let us denote by $\Theta_s$ the normalized derivative of the angular process i.e. $\Theta_s := \dot{x}_s/|\dot{x}_s|$ and by $X_s$ the target process $X_s:=(t_s,\dot{t}_s, \Theta_s)$. By Proposition \ref{myprop}, under the hypothesis that $H^3 \notin \mathbb L^1$, the clock $C_s$ is almost surely divergent, hence the process $\Theta_s$ is ergodic on the sphere.
Let us consider two independant versions $X_s^i=(t_s^i, \dot{t}_s^i,\Theta_s^i)$ of the process $X_s$ starting from two distinct points $X_0^i=(t_0^i, \dot{t}_0^i,\Theta_0^i)$ for $i=1,2$. By the above Remark \ref{rem.indepcoupling},  there exists two random times (finite almost surely) $T_1$ and $T_2$ such that 
\[
(t_{T_1}^1, \dot{t}_{T_1}^1)=(t_{T_2}^2, \dot{t}_{T_2}^2) \quad a.s.
\]
Now define a new process $({\Theta'}_{s}^2)_{s \geq 0}$, such that ${\Theta'}_{s}^2$ coincides with ${\Theta}_{s}^2$ on the time interval $[0, T_2]$ and such that the future trajectory $({\Theta'}_{s}^2)_{s \geq T_2}$ is the reflection of $({\Theta}_{s}^1)_{s \geq T_1}$ with respect to the median hyperplan between the points $\Theta_{T^1}^1$ and $\Theta_{T^2}^2$. The new process ${X'}^2_s:=(t_s^2, \dot{t}_s^2, {\Theta'}_s^2)$ is again a version of the target process. Moreover, by ergodicity of the spherical process, the first time $T^*$ when the process $({\Theta}_{s}^1)_{s \geq T_1}$ intersects the median big circle between $\Theta_{T^1}^1$ and $\Theta_{T^2}^2$ is finite almost surely, and one has naturally ${X'}^2_{T_2+T^*}={X}_{T_1+T^*}^1$ almost surely, hence the result.
\end{proof}

Let us now remark that the generator $L$ of the full diffusion is equivariant under the action of $(\mathbb R^d,+)$ by translation. Indeed, if $L_X$ denotes the generator of the subdiffusion $X_s=(t_s, \dot{t}_s,\Theta_s)$, we have 
\[
L = \dot{x} \, \partial_x + L_X = \left( \frac{\dot{x}}{|\dot{x}|} \times |\dot{x}|\right) \, \partial_x + L_X = \left(\Theta \times \frac{\sqrt{\dot{t}^2-1}}{\alpha(t)} \right)\,\partial_x + L_X,
\]
and thus the action of $\mathbb R^d$ by translation on the variable $x$ leaves $L$ unchanged. In other words, we are in position to apply the d\'evissage scheme with 
\[
X_s=(t_s, \dot{t}_s,\dot{x}_s/|\dot{x}_s|) \in X:=(0, +\infty) \times [1, +\infty) \times \mathbb S^{d-1}, \quad Y_s:=x_s \in G=(\mathbb R^d,+).
\]
By Theorem \ref{theo.trivial}, we deduce that almost surely
\[
\text{Inv}(t_s, \dot{t}_s,\dot{x}_s/|\dot{x}_s|,x_s) = \text{Inv}(t_s, \dot{t}_s,\dot{x}_s/|\dot{x}_s|) \vee \sigma(x_{\infty}).
\]
Finally, using Proposition \ref{pro.etrivial}, we can then conclude that 
\[
\text{Inv}(t_s, \dot{t}_s,\dot{x}_s/|\dot{x}_s|,x_s) = \sigma(x_{\infty}).
\]

\paragraph{The non flat cases with $H^3 \notin \mathbb L^1$.} 

We treat here the case where the Riemannian fiber $M$ is $\mathbb S^d$ or $\mathbb H^d$. As before we suppose that $H^3 \notin \mathbb L^1$ and $I(\alpha) <+\infty$ and recall that, by the above Proposition \ref{myprop}, this is equivalent to $C_{\infty} = +\infty$ and $D_{\infty} < +\infty$ almost surely.  The idea to establish the remaining cases of  Theorem \ref{theo.finitehori} is to apply the devissage method after having lifted the dynamics of $(x_s, \dot{x}_s/ \vert x_s\vert) \in T^1M$ into the group $G$ of isometries of $M$. As seen previously, recall that $T^1M= G/K$ where $K=\mathrm{SO}(d-1)$, and the group $G$ is $\mathrm{SO}(d+1)$ or $\mathrm{PSO}(1,d)$ when $M$ is $\mathbb S^d$ or $\mathbb H^d$ respectively. Recall also that the projection of $G$ onto $T^1M$ is given by $g \mapsto (g(e_0),g(e_1))$. 
In order to lift the dynamics of the spatial components of the relativistic diffusion to $G$, let us introduce the following matrices belonging to $\mathrm{Lie}(G)$,
\[
H_0 :=e_0e_1^{*} - \kappa e_1 e_0^*, \ \mathrm{and} \  V_i:= e_1 e^{*}_i - e_i e^{*}_1, \quad  i=2,\dots,d,
\]
where $ \kappa =\pm1$ denotes again the curvature of the fiber $M$. Then consider a $(d-1)-$dimensional Brownian motion $(B^{i}_s)_{i=2, \dots d}$ independent of the martingale $M^{\dot{t}}$ and define the process $(g_s)_{s\geq 0}$ with values in $G$ as the solution of the stochastic differential equation
\begin{equation}\label{dyn.g}
dg_s= \dot{D}_s g_s H_0 ds + \sigma \sqrt{\dot{C}_s} \sum_{i=2}g_s V_i \circ dB_s^{i}.
\end{equation}
Setting $(x_s, \dot{x}_s /\vert \dot{x}_s \vert ):=(g_s(e_0),g_s(e_1))$, one can easily check that $(x_s, \dot{x}_s /\vert \dot{x}_s \vert )_{s \geq 0}$ has indeed the dynamics given by equation \eqref{eqn.xdot}. This means that the relativistic diffusion $(t_s, \dot{t}_s, \dot{x}_s/\vert \dot{x}_s \vert, x_s)_{s \geq 0}$
can be obtained by projecting a diffusion $(t_s, \dot{t}_s, g_s)_{s\geq 0}$ with values in $\R \times [1,+\infty[ \times G$.
The idea behing the proof is now to split the dynamics of $(g_s)_{s \geq 0}$  into a converging part and an ergodic one. For that, let us define an auxiliary process $(b_s)_{s\geq0}$ as the solution of the stochastic differential equation 
\begin{equation}\label{dyn.b}
db_s = \sigma \sqrt{\dot{C}_s} \sum_{i=2}^{d} b_s V_i \circ dB^{i}_s.
\end{equation}
By definition the process $(b_s)_{s \geq 0}$ stabilizes $e_0$ and so belongs to $\mathrm{SO}(d)= \{ g \in G; \ g(e_0)=e_0\}$. Define now $u_s:= g_s b_s^{-1}$, so that $g_s = u_sb_s$. Using Equations \eqref{dyn.g} and \eqref{dyn.b}, it follows that
\begin{equation}\label{eq.u}
d u_s = \dot{D}_s u_s b_s H_0 b_s^{-1} ds.
\end{equation}
The relativistic diffusion $(t_s, \dot{t}_s, \dot{x}_s/\vert \dot{x}_s \vert, x_s)_{s \geq 0}$ can thus be obtained as the projection of the new diffusion $(t_s, \dot{t}_s, b_s, u_s)_{s \geq 0}$, which takes values in $\R \times [1,+\infty[ \times \mathrm{SO}(d) \times G$. The projection map is given by
\[
\begin{matrix}
\hat{\pi} :&  \R \times [1,+\infty[ \times \mathrm{SO}(d) \times G &\longmapsto& \R \times [1,+\infty[ \times T^1 M \\ 
&(t, \dot{t}, u, b) &\longmapsto & (t, \dot{t}, u(e_0), ub(e_1) ).
\end{matrix}
\]
In order to apply the devissage scheme for the new diffusion $(t_s, \dot{t}_s, b_s, u_s)_{s \geq 0}$, let us first show the following result.
\begin{prop}\label{prop.predeviss}
Under the condition $C_\infty =+\infty$ and $D_\infty < +\infty$ the two following points occur
\begin{enumerate}
\item The process $u_s$ converges almost surely to a random point $u_{\infty} \in G$.
\item The Poisson boundary of $(t_s, \dot{t}_s, b_s)_{s \geq 0}$ is trivial.
\end{enumerate}
\end{prop}

\begin{proof}
Let us start with the first point. By definition and by Equation \eqref{eq.u}, the process $(u_s)_{s \geq 0}$ is a $C^1$ path in $G$ and for any left invariant metric $\Vert \cdot \Vert$ on $G$, the total length of $(u_s)_{s \geq 0}$ is $\int_0^{+\infty} \dot{D}_s \Vert b_s H_0 b_s^{-1} \Vert ds$. But since the process $(b_s)_{s \geq 0}$ evolves in a compact set $\mathrm{SO}(d)$, one has the uniform estimate $\sup_{s\geq 0} \Vert b_s H_0 b_s^{-1} \Vert <+\infty$. Then, since $D_\infty < +\infty$ almost surely, it follows that the total length of $(u_s)_{s \geq 0}$ is finite, and  by completeness, $u_s$ converges almost surely.
Let us now give the proof of the second point of the statement. 
The arguments  are similar to those in the proof of Proposition 3.6.3 in \cite{kinetic}, which we adapt here to our context. Let us denote by $C_s^{-1}$ the generalized inverse of the clock $C_s$ and define the process $(\hat{b}_r)_{r\geq 0}$ in $\mathrm{SO}(d)$ as the solution of the stochastic differential equation
\[
d \hat{b}_r = \sigma \sum_{i=2}^{d} \hat{b}_r V_i \circ d\hat{B}^{i}_r, \quad \text{where} \;\; \hat{B}_r:= \int_0^{C^{-1}_r} \sqrt{\dot{C}_s}dB_s,
\]
i.e. $\hat{B}_r$ is a  $(d-1)-$dimensional Brownian motion independent of the martingale $M^{\dot{t}}$. Then, by definition, one gets that $b_s = \hat{b}_{C_s}$ or in other words, $(b_s)_{s \geq 0}$ is a time changed of the ergodic left invariant diffusion $(\hat{b}_r)$, taking values in $\mathrm{SO}(d)$ and  independent of the temporal process $(t_s, \dot{t}_s)_{s \geq 0}$.
As in the proof of Proposition \ref{pro.etrivial} above, following \cite{cranston}, recall that to get the triviality of the Poisson boundary, it is sufficient to construct a shift coupling between any two versions of the diffusion started from different initial conditions. 
So let $\big(t_s^{1}, \dot{t}_s^{1}, b^{1}_s\big)_{s \geq 0}$ and $\big(t_s^{2}, \dot{t}_s^{2}, b^{2}_s\big)_{s \geq 0}$ be two independant versions of the $\big(t_s,\dot t_s,b_s\big)$-diffusion and starting from 
$\big(t_0^{1}, \dot{t}_0^{1}, b_0^{1}\big)$ and $\big(t_0^{2}, \dot{t}_0^{2}, b_0^{2}\big)$ respectively.  We then know from Remark \ref{rem.indepcoupling} above that there exists coupling times $T_1$ and $T_2$ that are finite $\Prob_{(t_0^1, \dot{t}_0^1, b^1_0)}-$almost surely and $\Prob_{(t_0^2, \dot{t}_0^2, b^2_0)}-$almost surely respectively, and such that one can modify the process $\big(t_s^1, \dot{t}_s^1, b^1_s\big)_{s \geq 0}$ to get 
\[
\big(t_{T_1+s}^1, \dot{t}_{T_1+s}^1\big) = \big(t_{T_2+s}^2, \dot{t}_{T_2+s}^2\big), \quad \forall s \geq 0.
\]
In particular, for all $s \geq 0$, we have 
\[
\int_{T_1}^{T_1+s}\frac{dr}{|\dot t_r^1|^2-1} = \int_{T_2}^{T_2+t}\frac{dr}{|\dot t_r^2|^2-1}.
\]
Since the left invariant Brownian motion $\hat{b}_s$ in $\mathrm{SO}(d-1)$ is ergodic, one can find two Brownian motions $(\hat{b}^1(s))_{s\geq 0}$ and $(\hat{b}^2(s))_{s\geq 0}$ on $\mathrm{SO}(d-1)$, started from $b^1_{T_1}$ and $b^2_{T_2}$ respectively, which are independant of the temporal subdiffusions and which couple almost surely in finite time. Then, the processes $Z^i_s$, $i=1,2$, defined by the formulas
\begin{equation*}
Z^i_s= \left \lbrace \begin{array}{ll}
\displaystyle{\big(t_s^i, \dot{t}_s^i, b^i_s\big)}, & \textrm{for } 0\leq s \leq T_i, \\
\\
\displaystyle{ \left(t_s^i, \dot{t}_s^i, \hat{b}^i \left( \int_{T_i}^s\frac{dr}{\vert\dot t_r^i|^2 -1}\right)\right)}, & \textrm{for }s\geq T_i,
\end{array}\right.
\end{equation*}
have the laws of $\big(t_s,\dot t_s,b_s\big)$-diffusions started from $\big(t_0^1, \dot{t}_0^1, b^1_0\big)$ and $\big(t_0^2, \dot{t}_0^2, b^2_0\big)$, respectively and will couple in finite time almost surely, hence the result.
\end{proof}

Having established Proposition \ref{prop.predeviss}, we can now apply the d\'evissage scheme to the hypoelliptic diffusion$(t_s, \dot{t}_s, b_s, u_s)_{s\geq0}$.  Namely, setting $X_s:=(t_s, \dot{t}_s, b_s)$ and $Y_s=u_s$, Theorem \ref{theo.trivial} ensures that
\[
\mathrm{Inv}((t_s, \dot{t}_s, b_s, u_s)_{s\geq 0}) = \sigma(u_\infty), \quad \Prob_{(t_0,\dot{t}_0, b_0, u_0)}-\mathrm{a.s}.  
\]
In other words, we have identified the Poisson boundary of the lifted diffusion.
We need now to come back to the relativistic diffusion and check that, under the  projection $\hat{\pi}$, the only invariant information remaining is $u_\infty (e_0)$, the first column vector of the matrix $u_\infty$. For this, let us prove the following Lemma.

\begin{lemm}\label{invk}
Let $k\in \mathrm{SO}(d)$, the law of $(u_s)_{s \geq 0}$ under $\mathbb P_{(t,\dot{t},uk,k^{-1})}$  is the one of $(u_s k)_{s \geq 0}$ under $\mathbb P_{(t, \dot{t}, u, \mathrm{Id})}$.
\end{lemm}
\begin{proof}
Consider a trajectory $(t_s, \dot{t}_s, u_s, b_s)_{s \geq 0}$ starting at  $(t,\dot{t},uk,k^{-1})$. Set $u'_s:= u_s k^{-1}$ and $b'_s:= kb_s$. Thus, $u'_0= u$, $b'_0 = \mathrm{Id}$ and 
\begin{align*}
d b'_s = \sigma \sqrt{\dot{C}_s} \sum_{i=2}^d b'_s V_i \circ dB^{i}_s, \quad
d u'_s =  D_s u'_s b'_s H_0 (b'_s)^{-1} ds. 
\end{align*}
Solving the same stochastic differential equation / martingale problem,  $(t_s, \dot{t}_s, u'_s, b'_s)_{s\geq0}$ has then the same law as $(t_s, \dot{t}_s, u_s, b_s)_{s \geq 0}$ under $\mathbb P_{(t,\dot{t}, u, \mathrm{Id})}$. 
\end{proof}
Let us now show that the Poisson boundary of the relativistic diffusion coincides with $\sigma(x_{\infty})$ up to negligeable sets. Equivalently, we have to show that any bounded harmonic function can be written as the expectation of a bounded measurable functional of $x_{\infty}$. So let us consider a function $f: \R \times [1,+\infty [ \times T^1M \to \R$ which is bounded and harmonic with respect to the generator of the relativistic diffusion $(t_s, \dot{t}_s, \dot{x}_s/\vert \dot{x}_s \vert, x_s)_{s\geq 0}$. Then, $f\circ \pi : \R \times [1,+\infty [ \times G \times \mathrm{SO}(d) \to \R$ is a bounded harmonic function with respect to the generator of the lifted diffusion  $(t_s, \dot{t}_s, \dot{x}_s/\vert \dot{x}_s \vert, u_s, b_s)_{s \geq 0}$. Since the Poisson boundary of the lifted diffusion coincides with $\sigma(u_{\infty})$ almost surely, there exists a function $F: G \to \R$ measurable and bounded such that for all $(t,\dot{t}, u, b) \in \R \times [1,+\infty [ \times G \times \mathrm{SO}(d)$ 
\[
f \circ \pi (t,\dot{t}, u, b) = \mathbb{E}_{(t,\dot{t}, u, b)} [ F(u_\infty )].
\]
Fix now $(t, \dot{t}, \frac{\dot{x}}{\vert \dot{x} \vert}, x) \in \R \times [1, +\infty[ \times T^1 M$ and $u\in G$ such that 
\[
\pi(t, \dot{t}, u, \mathrm{Id} ) = (t, \dot{t}, \frac{\dot{x}}{\vert \dot{x} \vert}, x).
\]
Since for all $k \in \mathrm{SO}(d)$, $\pi (t, \dot{t}, uk, k^{-1}) = \pi(t, \dot{t}, u, \mathrm{Id} )$ it goes,
\[
f\left (t, \dot{t}, \frac{\dot{x}}{\vert \dot{x} \vert}, x\right )= \mathbb{E}_{(t, \dot{t}, uk, k^{-1})} [F(u_\infty)].
\]
By Lemma \ref{invk}, $u_\infty$ has under $\mathbb P_{(t, \dot{t}, uk, k^{-1})}$ the same law as $u_\infty k$ under $\mathbb P_{(t, \dot{t}, u, \mathrm{Id})}$ and so, for all $k\in SO(d)$ it comes
\[
f\left (t, \dot{t}, \frac{\dot{x}}{\vert \dot{x} \vert}, x\right )= \mathbb{E}_{(t, \dot{t}, u, \mathrm{Id})} [ F(u_\infty k) ].
\]
Integrating $k$ with respect to the Haar measure of $\mathrm{SO}(d)$, we get
\[
f\left (t, \dot{t}, \frac{\dot{x}}{\vert \dot{x} \vert}, x\right )=  \mathbb{E}_{(t, \dot{t}, u, \mathrm{Id})} \left [ \int_{\mathrm{SO}(d)} F(u_\infty k) dk \right ].
\]
Consider now a measurable section $S: M = G/\mathrm{SO}(d) \to G$ such that $S \circ \pi = \mathrm{Id} $ on $G$. Set, for $x \in M$ 
\[
\tilde{F}(x) := \int_{\mathrm{SO}(d)} F(S(x)k) dk.
\]
Then, defining $x_\infty:= u_\infty(e_0)$ one obtains finally,
\[
f\left (t, \dot{t}, \frac{\dot{x}}{\vert \dot{x} \vert}, x\right ) =  \mathbb{E}_{(t, \dot{t}, u, \mathrm{Id})} [\tilde{F}(x_\infty)],
\]
hence the result.

\subsubsection{Robertson--Walker space-times with infinite horizon}\label{sec.RWinfinite}
This last section is devoted to the proofs of Theorems \ref{theo.infhorEucli}, \ref{theo.infhorSpher} and \ref{theo.infhorHyp}. We thus assume that the horizon is infinite, \emph{i.e} $I(\alpha) = + \infty$.
In view of the hypotheses stated in Section \ref{sec.geoRW}, the non-integrability of its inverse implies that the torsion function $\alpha$ has polynomial growth, i.e. we have $\lim_{t \to + \infty} H(t) \times t = c$ with $c\in (0,1)$. This implies also that $H^3 \in \mathbb L^1$, and therefore by Proposition \ref{myprop}, we are dealing with a case where $C_{\infty}<+\infty$ and $D_{\infty}=+\infty$ almost surely. In this setting, the process $\dot{t}_s$ is transient almost surely and goes to infinity at an exponential rate.

\begin{prop}[Proposition 4.3 in \cite{angst2016}] \label{prop.asympdott}
If the torsion function $\alpha$ is such that its logarithmic derivative $H=\alpha'/\alpha$ satisfies $\lim_{t \to + \infty} H(t) \times t = c \in (0,1)$, then, almost surely as $s$ goes to infinity, we have
\[
\frac{1}{s}\log \dot{t}_s \underset{s \to +\infty}{\longrightarrow} \frac{d-1}{2}\frac{\sigma^2}{1+c}, \quad \frac{1}{s}\log \alpha(t_s) \underset{s \to +\infty}{\longrightarrow}\frac{d-1}{2}\frac{\sigma^2 c}{1+c}.
\]
\end{prop}

\newpage
\paragraph{The flat case.}
Let us first give the proof of Theorem \ref{theo.infhorEucli} and determine the Poisson boundary of the relativistic Brownian motion in the case where the fiber is Euclidean. We will proceed by using twice the d\'evissage method. In virtue of Proposition 4.7 and its proof in \cite{angst2016}, the asymptotic behavior of the relativistic diffusion $(\xi_s, \dot{\xi}_s)=(t_s,\dot{t}_s, x_s, \dot{x}_s)$ is then the following: 
\begin{itemize}
\item the temporal subdiffusion $(t_s,\dot{t}_s)$ is transient.
\item the process $(t_s,\dot{t}_s, \dot{x}_s/|\dot{x}_s|)$ is also a subdiffusion of the whole process.
\item there exists a Brownian motion $\widetilde{\Theta}_s$ with value in $\mathbb S^{d-1}$, which is independant of the temporal subdiffusion and such that normalized angular derivative writes $\dot{x}_s/|\dot{x}_s| = \widetilde{\Theta}(C_s)$. 
\end{itemize}
To simplify the expression let us write $\Theta_s:=\dot{x}_s/|\dot{x}_s|$, i.e. $\Theta_s=\widetilde{\Theta}(C_s)$. Since the clock $C_s$ converges almost surely as $s$ goes to infinity, the process $\Theta_s$ converges almost surely to a point $\Theta_{\infty}$ in $\mathbb S^{d-1}$.
The diffusion $(t_s,\dot{t}_s, \Theta_s)$ belongs to $(0,+\infty) \times [1, +\infty) \times \mathbb S^{d-1}$ which we identify with $(0,+\infty) \times [1, +\infty) \times \mathrm{SO}(d-1)/\mathrm{SO}(d-2)$. Being a time-changed spherical Brownian motion, the dynamics of $\Theta_s$ is equivariant by the action of $\mathrm{SO}(d-1)$ and applying Theorem \ref{theo.homo} with $X_s:=(t_s,\dot{t}_s)$ and $Y_s=\Theta_s$, we thus get that almost surely
\[
\text{Inv}((t_s, \dot{t}_s,\Theta_s)_{s \geq 0}) = \text{Inv}((t_s, \dot{t}_s)_{s \geq 0})  \vee \sigma( \Theta_{\infty}).
\]
Since $\text{Inv}((t_s, \dot{t}_s)_{s \geq 0})$ is trivial by Proposition \ref{pro.temptrivial}, we get that 
\[
\text{Inv}((t_s, \dot{t}_s,\Theta_s)_{s \geq 0}) =  \sigma( \Theta_{\infty}).
\]
Now define
\[
\delta_s := x_s - \Theta_s \int_{t_0}^{t_s} \frac{du}{\alpha(u)}.
\]
Then, since $x_s= \int_0^s \dot{x}_u du = \int_0^{s} \Theta_u \dot{D}_u du$, one can decompose $\delta_s$ as
\begin{equation}\label{eq.defdelta}
\delta_s = \int_0^{s} \left ( \Theta_u - \Theta_{\infty} \right ) \dot{D}_u du + \big (  \Theta_{\infty}- \Theta_s \big)\int_0^{s} \dot{D}_u du + \Theta_s \left ( D_s-\int_{t_0}^{t_s} \frac{du}{\alpha(u)} \right).
\end{equation}
By the asymptotics estimates of Proposition \ref{prop.asympdott} above (see also Proposition 4.3 of \cite{angst2016} for more details), one can check that each of the three terms on the right hand side of the last Equation \eqref{eq.defdelta} converges almost surely as $s$ goes to infinity, so that $\delta_s$ converges to a random $\delta_\infty \in \R^d$. 
Moreover, from Equation \eqref{eq.defdelta} again, since the dynamics of $\delta_s$ only depends on $(t_s, \dot{t}_s,\Theta_s)$, it is clear that the generator of the diffusion $(t_s, \dot{t}_s,\Theta_s,\delta_s)$ is equivariant under the action of $(\mathbb R^d,+)$ by translation on the variable $\delta$. Therefore, we are in position to apply Theorem \ref{theo.trivial} with this time $X_s:=(t_s, \dot{t}_s,\Theta_s)$ and $Y_s:=\delta_s$. We thus get
\[
\text{Inv}((t_s, \dot{t}_s,\Theta_s, \delta_s)_{s \geq 0}) = \text{Inv}((t_s, \dot{t}_s,\Theta_s)_{s \geq 0}) \vee \sigma( \delta_{\infty})
\]
and thus, from the first d\'evissage result of the proof
\[
\text{Inv}((t_s, \dot{t}_s,\Theta_s, \delta_s)_{s \geq 0}) =  \sigma( \Theta_{\infty},\delta_{\infty}).
\]
To conclude, remark that the relativistic diffusion $(t_s, \dot{t}_s,\Theta_s,x_s)$ is in bijection with $(t_s, \dot{t}_s,\Theta_s,\delta_s)$ so that 
\[
\text{Inv}((t_s, \dot{t}_s,\Theta_s, x_s)_{s \geq 0}) = \text{Inv}((t_s, \dot{t}_s,\Theta_s,\delta_s)_{s \geq 0}) = \sigma(\Theta_{\infty}, \delta_{\infty}).
\]

\paragraph{The spherical case.}
We now give the proof of Theorem \ref{theo.infhorSpher} for a warped product with spherical fiber.
We proceed as before, i.e. we lift the relativistic diffusion $(t_s, \dot{t}_s, x_s, \dot{x}_s/\vert \dot{x}_s \vert )_{s\geq 0}$ to a diffusion $(t_s, \dot{t}_s, g_s)_{s\geq 0}$ where $g_s\in \mathrm{SO}(d+1)$ has the following dynamics
\[
dg_s = \dot{D}_s g_s H_0 ds +\sigma \sqrt{\dot{C}_s} \sum_{i=2}^d g_s V_i \circ dB_s^{i},
\]
with $H_0:= e_0 e^{*}_1 - e_1 e^{*}_0$, $V_i= e_1 e^{*}_i - e_i e^{*}_1$. Then by definition $x_s:= g_s(e_0)$ and $\dot{x}_s/\vert \dot{x}_s \vert := g_s (e_1)$.
Under the hypotheses of Theorem \ref{theo.infhorSpher}, we have $C_{\infty}<+\infty$ and $D_{\infty}=+\infty$ almost surely. 
The idea is then again to split the converging components from the diverging ones in the dynamics of $g_s$. As previously, the diverging clock $D_s$ is a function of the all trajectory $(\dot{t}_u, t_u)_{0\leq u \leq s}$ but not only of $(\dot{t}_s, t_s)$. To skip this difficulty, one defines the new clock $A(t):=\int_{t_0}^{t}du/\alpha(u)$ and, thanks to Lemma 4.6 of \cite{angst2016}, we know that $A(t_s)-D_s$ converges almost surely as $s$ goes to infinity. 
Let us now define
\[
\hat{u}_s := \exp \left ( A(t_s) H_0 \right ) =  \left ( \begin{matrix}  \cos( A(t_s)) &- \sin( A(t_s) ) & 0 \\ \sin( A(t_s) ) & \cos( A(t_s) ) & 0 \\ 0 & 0 & \mathrm{Id}  \end{matrix} \right ).
\]
Next we set $\hat{b}_s := g_s \hat{u}_{s}^{-1}$, and doing so we get rid of the diverging part in the dynamics of $g_s$. Precisely we have the following lemma.

\begin{lemm}
The process $(\hat{b}_s)_{s\geq0}$ converges almost surely to an asymptotic random variable $\hat{b}_{\infty}$.
\end{lemm}
\begin{proof}
Let us introduce the process
\begin{align*}
u_s := \exp \left ( D_s  H_0 \right) = \left ( \begin{matrix}  \cos( D_s ) & -\sin( D_s ) & 0 \\ \sin( D_s) & \cos( D_s) & 0 \\ 0 & 0 & \mathrm{Id}  \end{matrix} \right ).  
\end{align*}
By  definition we have then $u_s \hat{u}_{s}^{-1} = \exp \left ( \big( D_s - A(t_s) \big) H_0 \right )$ and as recalled above, by Lemma 4.6 of \cite{angst2016}, the difference $D_s - A(t_s)$ converges almost surely and so does $u_s \hat{u}_{s}^{-1}$. 
Now, we can define $b_s := g_s u^{-1}_s$ which then solves de stochastic differential equation
\begin{align*}
d b_s = \sigma \sqrt{\dot{C}_s} \sum_{i=2}^d  b_s u_s V_i u^{-1}_s \circ dB^i_s.
\end{align*}
Since $u_s \in \mathrm{SO}(d+1)$ the term $u_s V_i u^{-1}_s$ is bounded. Therefore, since the clock $C_s$ converges almost surely, one deduce that $b_s$ converges almost surely as $s$ goes to infinity to a random variable $b_\infty \in G$. 
Finally, since $\hat{b}_s = g_s \hat{u}_s^{-1} = b_s u_s \hat{u}^{-1}_s$, one deduces that $\hat{b}_s$ also converges almost surely to some $\hat{b}_{\infty} \in G$. 
\end{proof}

Note that since $\hat{u}_s$ is a only function of $t_s$, the diffusion process $(t_s, \dot{t}_s, g_s)_{s\geq 0}$ is in fact in bijection with the diffusion process $(t_s, \dot{t}_s, g_s \hat{u}_s^{-1})_{s\geq 0}=(t_s, \dot{t}_s, \hat{b}_s)_{s\geq 0}$, and so they have the same invariant sigma field. Moreover, the process $\hat{b}_s$ solves the stochastic differential equation
\[
d \hat{b}_s =  \left (D_s - \frac{\dot{t}_s}{\alpha(t_s)} \right ) \hat{b}_s \hat{u}_s H_0 \hat{u}^{-1}_s ds + C_s \sum_{i=2}^d \hat{b}_s \hat{u}_s V_i \hat{u}_s^{-1} \circ dB_s^{i},
\]
whose dynamics only depends on $(\dot{t}_s, t_s)$. Therefore, the generator of the diffusion $(t_s, \dot{t}_s, \hat{b}_s)_{s\geq 0}$ is equivariant under the action of $G$ by translation and one can apply the devissage Theorem \ref{theo.trivial} with $X_s:=(t_s, \dot{t}_s)$ and $Y_s:= \hat{b}_s$, associated with Proposition \ref{pro.temptrivial}, to conclude that almost surely 
\[
\mathrm{Inv}((t_s, \dot{t}_s, g_s)_{s\geq 0})  =\mathrm{Inv}((t_s, \dot{t}_s, \hat{b}_s)_{s\geq 0}) =\mathrm{Inv}((t_s, \dot{t}_s)_{s \geq 0}) \vee  \sigma(\hat{b}_{\infty})= \sigma(\hat{b}_{\infty}).
\]
We now need to come back to the process $(\dot{x}_s/ \vert \dot{x}_s \vert, x_s)$. Recall that 
\begin{align*}
x_s &= g_s(e_0) = \hat{b}_s \hat{u}_s (e_0)= \cos(A(t_s) ) \hat{b}_s (e_0) + \sin(A(t_s)) \hat{b}_s (e_1),  \\ 
\frac{\dot{x}_s}{\vert x_s \vert} &= g_s(e_1) =  \hat{b}_s \hat{u}_s (e_1) = -\sin(A(t_s) ) \hat{b}_s (e_0) + \cos(A(t_s)) \hat{b}_s (e_1).
\end{align*}
But we also have the representation
\begin{align*}
\hat{b}_s (e_0) &= \cos(A(t_s) ) x_s  - \sin(A(t_s) ) \frac{\dot{x}_s}{\vert x_s \vert}, \\ 
\hat{b}_s (e_1) &= \sin(A(t_s)) x_s + \cos(A(t_s) ) \frac{\dot{x}_s}{\vert x_s \vert}
\end{align*}
so that the relativistic diffusion $(t_s, \dot{t}_s, \dot{x}_s/\vert \dot{x}_s \vert, x_s)_{s\geq 0}$ is in bijection with $(t_s, \dot{t}_s, \hat{b}_s(e_0), \hat{b}_s (e_1) )_{s\geq 0}$. Therefore, we deduce that
\[
\mathrm{Inv} ((t_s, \dot{t}_s, \dot{x}_s/\vert \dot{x}_s \vert, x_s)_{s\geq 0}) = \mathrm{Inv} ((t_s, \dot{t}_s, \hat{b}_s(e_0), \hat{b}_s (e_1) )_{s\geq 0}).
\]
Finally, viewing the unitary tangent space $T^1\mathbb S^d$ as the homogeneous space $\mathrm{SO}(d+1)/\mathrm{SO}(d)$, one concludes, invoking Theorem \ref{theo.homo} that $\mathrm{Inv} ((t_s, \dot{t}_s, \hat{b}_s(e_0), \hat{b}_s (e_1) )_{s\geq 0})=\sigma( \hat{b}_{\infty}(e_0), \hat{b}_{\infty}(e_1))$. And thus almost surely, we have  
\[
\mathrm{Inv}((t_s, \dot{t}_s, \dot{x}_s/\vert \dot{x}_s \vert, x_s)_{s\geq 0}) =  \sigma ( b_\infty(e_0), b_\infty(e_1) ),
\]
hence the result. Note that $b_\infty(e_0)$ and $b_\infty(e_1)$ are two orthogonal vectors of norm one, which generate a random plan. This intersection of this plan with we sphere is precisely the random big circle to which the process $(x_s)_{s \geq 0}$ is asymptotic, as illustrated in Figure \ref{fig.sphere}.

\paragraph{The hyperbolic case.}
Finally, let us give the proof of Theorem \ref{theo.infhorHyp} dealing with the case of an hyperbolic fiber.
To do so, we lift again the relativistic diffusion to the orthonormal frame bundle. 
Recall that $\Hyp^d$ is equivalently seen as the half pseudo-sphere of $\R^{1,d}$ and the homogenous space $\mathrm{PSO}(1,d)/\mathrm{SO}(d)$ and the relativistic diffusion $(\dot{t}_s, t_s, \dot{x}_s/\vert \dot{x}_s \vert , x_s)_{s \geq 0}$ can again be obtained as a projection, namely we can represent the spatial components as $x_s:= g_s (e_0)$, $\dot{x}_s/\vert \dot{x}_s \vert := g_s(e_1)$ where the diffusion $(\dot{t}_s, t_s, g_s)$ now takes values in $T\R^{*}_{+} \times \mathrm{PSO}(1,d)$ and solves the stochastic differential equation
\begin{align*}
d g_s = \dot{D}_s g_s H_1 ds +\sigma \sqrt{\dot{C}_s} \sum_{i=2}^{d} g_s V_i \circ dB_s^{i},
\end{align*}
with this time $H_1:= e_0 e^{*}_1 + e_1 e^{*}_0$ and $V_i:= e_i e^{*}_1 - e_1 e^{*}_i$, $i=2,\dots,d$. As above, under the hypotheses of Theorem \ref{theo.infhorHyp}, we have $C_{\infty}<+\infty$ and $D_\infty=+\infty$ almost surely. The idea is again to separate the converging part and the diverging part in the dynamics of $g_s$. For that, we introduce the process $(u_s)_{s \geq 0}$ with values in $K:=\mathrm{SO}(d)$ solving
\[
d u_s = \sigma \sqrt{\dot{C}_s} \sum_{i=2}^d u_s V_i \circ dB_s^{i}.
\]
Since $C_\infty < +\infty$, the process $u_s$ converges almost surely as $s$ goes to infinity to a  random point $u_\infty \in K$. 
Let us now consider the Iwasawa decomposition $NAK$ of $G= \mathrm{PSO}(1,d)$, see section \ref{sec.iwa} before, and decompose the process $g_s$ with values in $\mathrm{PSO}(1,d)$ as $g_s=n_s a_s k_s$, where  $n_s \in N$, $a_s \in A$ and $k_s \in K$ obbey the following dynamics
\begin{align*}
d k_s &= - \dot{D}_s \sum_{i=2}^d (e^{*}_{i} k_s u_s e_1 ) V_i k_s ds, \\
d a_s &=  \dot{D}_s  (e^{*}_{1} k_s u_s e_1) a_s H_0 ds, \\ 
d n_s &= \displaystyle{\dot{D}_s \exp \left({- \int_{0}^{s} \dot{D}_r (e^{*}_1 k_r u_r e_1) dr }\right) \sum_{i=2}^d (e^{*}_i k_s u_s e_1 ) (H_i + V_i ) ds.}
\end{align*}
Now consider the one dimensional process $(v_s)_{s \geq 0}$ with values in $[-1,1]$, defined as $v_s := e_{1}^{*}k_s u_s e_1$, and which solves the stochastic differential equation
\[
d v_s = \dot{D}_s (1- v_s^2) ds + \sigma \sqrt{\dot{C}_s} \sqrt{1- v_s^2} dW_s - \sigma^2 \dot{C}_s \frac{d-1}{2} v_s ds,
\]
where $W_s$ is the standard Brownian motion such that
\[
\sqrt{1- v_s^2} dW_s = \sum_{i=2}^d e_{1}^{*} k_s u_s e_i  dB_s^{i}.
\]

\begin{lemm} \label{lem.v}
The process $v_s$ converges almost surely to $1$ as $s$ goes to infinity. 
\end{lemm}
\begin{proof}
By definition, we have
\begin{align}
v_s = v_0 + \int_0^s \dot{D}_u (1- v_u^2) du +\sigma \int_0^s \sqrt{\dot{C}_u} \sqrt{1- v_u^2} dW_u - \sigma^2 \frac{d-1}{2} \int_0^s \dot{C}_u v_u du.\label{eq1}
\end{align}
Since the clock $C_s$ converges almost surely as $s$ goes to infinity, the two last terms of Equation \eqref{eq1} also converge almost surely. Moreover,  $s\mapsto \int_0^s \dot{D}_u (1- v_u^2) du$ is increasing and bounded (recall that $|v_s|\leq 1$), and thus also converges almost surely as $s$ goes to infinity. Therefore, all the terms on the right hand side of Equation \eqref{eq1} are converging, and so does $v_s$. But since the clock $D_s$ diverges almost surely, the only possible limit of $v_s$ making the integral $\int_0^s \dot{D}_u (1- v_u^2) du$ convergent is $1$, hence the result. 
\end{proof}
Let us now introduce the two other auxiliary  processes $(\Theta_s)_{s \geq 0}$ and $(\tilde{\delta_s})_{s \geq 0}$ defined as 
\[
\Theta_s:= k_s u_s e_1 \in\mathbb S^{d-1}, \qquad  \tilde{\delta_s} :=  \tilde{\delta_0} + \int_0^s \dot{D}_u v_u du - A(t_s) \in \mathbb R,
\]
where as in the proof of the spherical case, we have set $A(t_s)=\int_{t_0}^{t_s} \frac{1}{\alpha(u)} du$.
Since, by Lemma \ref{lem.v}, the process $v_s$ goes to $1$ almost surely, then $\Theta_s$ converges almost surely to a deterministic limit, namely the basis vector $e_1$. Moreover, the process $(\tilde{\delta_s})_{s \geq 0}$ is also convergent. 

\begin{lemm}
The process $\tilde{\delta}_s$ converges almost surely as $s$ goes to infinity to a random $\tilde{\delta}_\infty $. 
\end{lemm}

\begin{proof}
In the proof of Lemma \ref{lem.v}, we have already shown that the integral $\int_0^s \dot{D}_u (1- v_u^2) du$ is almost surely convergent as $s$ goes to infinity. Since $|v_s|\leq 1$, and since $v_s$ tends to one, this implies that $ \int_0^s \dot{D}_u v_u du - D_s$ also converges almost surely. As already used in the proof of the spherical case, the Lemma is proven invoking Lemma 4.6 in \cite{angst2016} where it is shown that $D_s -  A(t_s)$ converges almost surely.
\end{proof}

We now go back to the Iwaswa decomposition of the process $g_s$ and concentrate on the nilpotent component. 

\begin{lemm}
The $N$-valued process $(n_s)_s$ converges almost surely to some asymptotic random variable $n_\infty \in N$. 
\end{lemm}

\begin{proof}
By Proposition \ref{prop.asympdott}, we have almost surely $\dot{D}_s = \exp \big ( \sigma^2 \frac{d-1}{2}\frac{1-c}{1+c} \, s + o(s) \big)$,  with $c\in (0,1)$. Thus $\int_0^{+\infty}\dot{D}_s e^{- \int_{0}^{s} \dot{D}_r v_r dr}$ is finite almost surely and, from the stochastic differential equation satisfied by $n_s$, one deduces that $\int_0^{+\infty} \vert \dot{n}_u \vert du <+\infty$. By completness of $N$, it follows that $n_s$ converges almost surely. 
\end{proof}

We can finally describe the Poisson boundary of the relativistic diffusion using iteratively the d\'evissage method. Let us remark that there is a bijective map between the relativistic diffusion $(\dot{t}_s, t_s, \dot{x}_s/\vert \dot{x}_s \vert , x_s)$ and  the process $(\dot{t}_s, t_s, \Theta_s, \tilde{\delta}_s, n_s)$, so that 
\[
\mathrm{Inv} ((t_s, \dot{t}_s, \dot{x}_s/\vert \dot{x}_s \vert, x_s)_{s\geq 0})  = \mathrm{Inv} ((t_s, \dot{t}_s, \Theta_s, \tilde{\delta}_s, n_s)_{s\geq 0}).
\]
So let us first consider the subdiffusion $(t_s, \dot{t}_s, \Theta_s)_{s \geq 0}$.
Using Proposition \ref{pro.temptrivial} and Theorem \ref{theo.homo} with $X_s=(\dot{t}_s, t_s)$ and $Y_s=\Theta_s \in \mathbb S^{d-1}$, the dynamics of $(X_s,Y_s)$ being equivariant under the action $\mathbb S^{d-1}$ by rotation,  we get that almost surely 
\[
\mathrm{Inv} ((t_s, \dot{t}_s, \Theta_s)_{s\geq 0}) = \mathrm{Inv} ((t_s, \dot{t}_s)_{s\geq 0}) \vee \sigma(\Theta_{\infty}) = \sigma(\Theta_{\infty}) .
\]
But since $\Theta_{\infty}=e_1$ is deterministic, we get in fact that $\mathrm{Inv} ((t_s, \dot{t}_s, \Theta_s)_{s\geq 0}) $ is trivial. Now consider the larger subdiffusion $(\dot{t}_s, t_s, \Theta_s, \tilde{\delta}_s)_{s \geq 0}$. Again, its dynamics is equivariant under the action of $(\mathbb R,+)$ by translation on the variable $\tilde{\delta}$. Setting this time $X_s=(t_s, \dot{t}_s, \Theta_s)$ and $Y_s=\tilde{\delta}_s \in \mathbb R$, and applying Theorem \ref{theo.trivial}, we get that almost surely 
\[
\mathrm{Inv} ((t_s, \dot{t}_s, \Theta_s, \tilde{\delta}_s)_{s\geq 0}) = \mathrm{Inv} ((t_s, \dot{t}_s, \Theta_s)_{s\geq 0}) \vee \sigma(\tilde{\delta}_{\infty}) = \sigma(\tilde{\delta}_{\infty}).
\]
Last, we can consider the full diffusion  $(\dot{t}_s, t_s, \Theta_s, \tilde{\delta}_s, n_s)_{s \geq 0}$. 
Setting now $X_s:=(\dot{t}_s, t_s, \Theta_s, \tilde{\delta}_s)$ and $Y_s=n_s$, the global dynamics being equivariant under the action of $N$, a last application of Theorem \ref{theo.trivial} yields almost surely
\[
\mathrm{Inv} ((t_s, \dot{t}_s, \Theta_s, \tilde{\delta}_s,n_s)_{s\geq 0}) = \mathrm{Inv} ((t_s, \dot{t}_s, \Theta_s, \tilde{\delta}_s)_{s\geq 0}) \vee \sigma(n_{\infty}) = \sigma(\tilde{\delta}_{\infty},n_{\infty}).
\]
To conclude the proof of Theorem \ref{theo.infhorHyp}, we need to verify that the limit random variable $(\tilde{\delta}_\infty, n_\infty)$ is a measurable function of $(\delta_{\infty}, \theta_\infty)$ appearing in the statement. Recall that by definition, 
\[
\delta_{\infty}=\lim_{s \to \infty} \delta_s, \; \text{where} \; \delta_s:= 1+ \mathrm{sinh}^{-1}(\vert \vec{x}_s \vert) - \int_{t_0}^{t_s} \frac{du}{\alpha(u)}, \;\; \text{and} \; \;
\theta_\infty := \lim_{s \to +\infty} \vec{x}_s / \vert \vec{x}_s \vert
\]
where $\vec{x}_s:= \sum_{i=1}^{d} x^{i} e_i$ is the $\R^d$ part (or spatial) part of $x_s \in \mathbb R^{1,d}$ and $|\vec{x}_s|$ is its Euclidean norm. Then, one can check that 
\[
\theta_\infty = n_{\infty}(e_0+e_1)/ (e_0^{*} n_{\infty}(e_0+e_1) ).
\] 
The map $n \in N \mapsto n(e_0+e_1)/ (e_0^{*} n(e_0+e_1))$ is a bijection from $N$ onto the set of point of the form $e_0 + \vec{x}$, $\vert \vec{x} \vert =1$ and $\vec{x}\neq -e_1$. This ensure that $n_\infty$ is a function of $\theta_\infty$. Moreover, we have 
\[
\tilde{\delta}_s = \int_0^{s} \dot{D}_u v_u du - \mathrm{sinh}^{-1}(\vert \vec{x}_s \vert) + \delta_s\] 
and one can easily check that $ \int_0^{s} \dot{D}_u v_u du - \mathrm{sinh}^{-1}(\vert \vec{x}_s \vert) $ converges almost surely to an asymptotic random variable which depends only of $\theta_\infty$ (this is because one passes from Iwasawa to polar coordinates in hyperbolic space). Thus $\tilde{\delta}_\infty$ is a measurable function of $(\theta_\infty, \delta_{\infty})$, and the invariant sigma field of the full relativistic diffusion is indeed generated by $(\theta_\infty, \delta_{\infty})$.

\end{document}